\documentclass[a4paper]{amsart}
\usepackage[latin9]{inputenc}
\usepackage{color}
\usepackage{verbatim}
\usepackage{amsthm}
\usepackage{amstext}
\usepackage{amssymb}
\usepackage{esint}
\usepackage[unicode=true,pdfusetitle,
 bookmarks=true,bookmarksnumbered=false,bookmarksopen=false,
 breaklinks=false,pdfborder={0 0 1},backref=section,colorlinks=false]
 {hyperref}

\makeatletter
\numberwithin{equation}{section}
\numberwithin{figure}{section}
\theoremstyle{plain}
\newtheorem{thm}{\protect\theoremname}
  \theoremstyle{definition}
  \newtheorem{defn}[thm]{\protect\definitionname}
  \theoremstyle{plain}
  \newtheorem{prop}[thm]{\protect\propositionname}
  \theoremstyle{remark}
  \newtheorem{rem}[thm]{\protect\remarkname}
  \theoremstyle{plain}
  \newtheorem{lem}[thm]{\protect\lemmaname}
  \theoremstyle{plain}
  \newtheorem{cor}[thm]{\protect\corollaryname}

\usepackage{amssymb,amsthm,amsmath,amsfonts,amscd}
\usepackage{amsaddr}
\usepackage{graphicx}
\usepackage{url}
\usepackage{color}
\usepackage[active]{srcltx}
\usepackage[matrix,arrow]{xy}
\usepackage{mathrsfs}
\usepackage{enumerate}
\usepackage{amsopn} 
\usepackage{bbm} 
\usepackage{ulem} 
\usepackage{cite}
\usepackage{hyperref}
\allowdisplaybreaks[1]
\usepackage[english]{babel}
\usepackage{bbm}
\usepackage{stmaryrd}
\usepackage{mdef}
\date{\today}

\makeatother

  \providecommand{\corollaryname}{Corollary}
  \providecommand{\definitionname}{Definition}
  \providecommand{\lemmaname}{Lemma}
  \providecommand{\propositionname}{Proposition}
  \providecommand{\remarkname}{Remark}
\providecommand{\theoremname}{Theorem}

\begin{document}

\title{Stochastic Scalar conservation laws driven by rough paths}
\begin{abstract}
We prove the existence and uniqueness of solutions to a class of stochastic scalar conservation laws with joint space-time transport noise and affine-linear noise driven by a geometric $p$-rough path. In particular, stability of the solutions with respect to the driving rough path is obtained, leading to a robust approach to stochastic scalar conservation laws. As immediate corollaries we obtain support theorems, large deviation results and the generation of a random dynamical system. 
\end{abstract}

\author{Peter K. Friz}

\address{Institut für Mathematik \\
Technische Universität Berlin \\
\&\\
Weierstrass\textendash{}Institut für Angewandte Analysis und Stochastik\\
Berlin, Germany\\
}

\email{friz@math.tu-berlin.de}

\author{Benjamin Gess}

\address{Department of Mathematics \\
University of Chicago \\
USA }

\email{gess@math.tu-berlin.de}



\thanks{\textbf{Acknowledgements}: P.K.F. has received funding from the European Research Council under the European Union's Seventh Framework Programme (FP7/2007-2013) / ERC grant agreement nr. 258237. B.G. has been partially supported by the RTG 1845 ``Stochastic Analysis with Applications in Biology, Finance and Physics'' and by the research project ``Random dynamical systems and regularization by noise for stochastic partial differential equations'' funded by the German Research Foundation. }

\maketitle

\section{Introduction}

We develop a rough path approach to a class of stochastic scalar conservation laws of the type
\begin{align}
du+\Div f(t,x,u)dt & =F(t,x,u)+\sum_{k=1}^{N}\L_{k}(x,u,\nabla u)\circ d\b_{t}^{k},\label{eq:full_smooth_noise-intro}\\
u(0) & =u_{0},\nonumber 
\end{align}
on $[0,T]\times\R^{d}$, where $f,F$ are continuous, $\L_{k}=\L_{k}(x,r,p)$ is affine-linear in $r,p$, that is
\[
\L_{k}(x,r,p)=p\cdot H_{k}(x)+r\nu_{k}+g_{k}(x)
\]
and $\b^{k}$ are real-valued Brownian motions. More generally, we will give meaning to \eqref{eq:full_smooth_noise-intro} when $\b$ is replaced by a general geometric $p$-rough path $\mathbf{z}$. The Stratonovich type solution to \eqref{eq:full_smooth_noise-intro} is then obtained by applying this to Brownian motion enhanced to a rough path. Further justification for the Stratonovich notation in \eqref{eq:full_smooth_noise-intro} is provided by a Wong-Zakai type limit theorem which becomes an immediate consequence of our main Theorem \ref{thm:RP_limit} (part iii) together with well-known rough paths convergence of piecewise linear (and many other) approximations to (enhanced) Brownian motion. For background on rough paths we refer to \cite{L98,LQ02,LCL07,FV10,HF12}. Roughly speaking the main results reads
\begin{thm}
\label{thm:RP_limit-1}Given sufficient regularity of $u_{0}$,$f$,$F$,$\L_{k}$ there exists a unique solution to
\begin{align}
du+\Div f(t,x,u)dt & =F(t,x,u)+\sum_{k=1}^{N}\L_{k}(x,u,\nabla u)\circ d\mathbf{z}_{t}^{k},\label{eq:full_smooth_noise-intro-1}\\
u(0) & =u_{0},\nonumber 
\end{align}
for every geometric rough path $\mathbf{z}$, in the following sense: There exists a unique $u=u^{\mathbf{z}}\in L^{\infty}([0,T]\times\R^{d})$ such that for every sequence $z^{n}\in C^{1}([0,T])$ with $z^{n}\to\mathbf{z}$ in rough path metric the (unique) weak entropy solutions to 
\[
\partial_{t}u^{n}+\Div f(t,x,u^{n})=F(t,x,u^{n})+\sum_{k=1}^{N}\L_{k}(x,u^{n},\nabla u^{n})\dot{z}_{t}^{n,k}
\]
converge to $u$ in in $L^{\infty}([0,T];L_{loc}^{1}(\R^{d}))$. The solution map $(\mathbf{z},u_{0})\mapsto u^{\mathbf{z}}$ is continuous in appropriate norms.
\end{thm}
As immediate benefits of taking a rough paths approach to stochastic scalar conservation laws and the resulting continuity of the solution map $ $$(\mathbf{z},u_{0})\mapsto u^{\mathbf{z}}$ one obtains support results, large deviation results and the generation of a random dynamical system as simple consequences (cf. \cite{FO13,CFO11} for details). Moreover, we should note that the range of driving signals covered by Theorem \ref{thm:RP_limit-1} goes far beyond Brownian motion. In particular, this includes fractional Brownian motion with Hurst parameter $H\in(\frac{1}{4},\frac{1}{2})$. 

In the construction of solutions we combine stability results from the theory of rough paths with stability of weak entropy solutions to space-time inhomogeneous scalar conservation laws. Due to the irregularity of the driving rough path $\mathbf{z}$, the coefficients of the corresponding inhomogeneous scalar conservation laws only satisfy little regularity (especially in the time variable) and related stability results have only recently been developed in \cite{LM11} in an $L^{1}$ framework. In order to combine such stability estimates with the $L^{\infty}$-stability estimates from rough paths theory we prove localized versions of the estimates derived in \cite{LM11}, thus leading to an $L_{loc}^{1}$ stability theory applicable to the situation at hand.

In the case of pure transport noise, i.e. 
\begin{align}
du+\Div f(u)dt & =\sum_{k=1}^{N}\L_{k}(x,\nabla u)\circ d\mathbf{z}_{t}^{k},\label{eq:rough_transport-1}\\
u(0) & =u_{0}\nonumber 
\end{align}
with $\L_{k}(x,p)=p\cdot H_{k}(x)$ we derive a rate on the convergence $u^{n}\to u$ proven in Theorem \ref{thm:RP_limit-1}. Roughly speaking, as a second main result we obtain
\begin{thm}
\label{thm:rate-1}For two rough paths \textbf{$\mathbf{z}^{1},\mathbf{z}^{2}$} let $u^{1},u^{2}$ be the corresponding solutions to \eqref{eq:rough_transport-1} with initial data $u_{0}^{1},u_{0}^{2}$ respectively. Then 
\begin{align*}
\sup_{t\in[0,T]}\|u^{1}(t)-u^{2}(t)\|_{L^{1}(\R^{d})}\le & \|u_{0}^{1}-u_{0}^{2}\|_{L^{1}(\R^{d})}+K\TV(u_{0}^{1})\rho(\mathbf{z}^{1},\mathbf{z}^{2}),
\end{align*}
where $K$ can be chosen locally uniformly with respect to $\mathbf{z}^{1},\mathbf{z}^{2}$ in rough path metric $\rho$.
\end{thm}
As it is well-known, scalar conservation laws of the general type \eqref{eq:full_smooth_noise-intro} do not belong to the class of (fully-)nonlinear PDE that may be treated by the theory of viscosity solutions. In particular, \eqref{eq:full_smooth_noise-intro} is out of reach of the results developed in \cite{CF09,CFO11,FO13,LS98,LS98-2,LS00,LS00-2}. Notably, our results are based on the notion of weak entropy solutions to \eqref{eq:full_smooth_noise-intro} rather than viscosity solutions. We should also point out that \eqref{eq:full_smooth_noise-intro} is of quasilinear type, so that the methods developed in \cite{DGT12,GLS12} and applicable to semilinear SPDE do not apply.

Many works have been devoted to the study of stochastic and random scalar conservation laws. Noise entering scalar conservation laws via randomness in the initial condition has been studied for example in \cite{AE95,S92-2,R98,B74}. For stochastic scalar conservation laws driven by additive noise, also including boundary value problems, we refer to \cite{N82,EKMS00,SS12,K03,VW09} and the references therein. The case of multiplicative noise, i.e. SPDE of the form
\[
du+\partial_{x}f(u)dt=g(x,u)dW_{t},
\]
has attracted considerable interest in recent years (cf. e.g. \cite{HR97,FN08,DV10,H13,CDK12,BVW13,DHV13}). All of the above mentioned works consider \textit{semilinear} stochastic scalar conservation laws in the sense that the diffusion coefficients do not depend on the derivative(s) of the solution. In contrast, in the recent works \cite{LPS12,LPS13} stochastic perturbations of the flux $f$ are considered, which in general leads to SPDE of the type
\[
du=\sum_{k=1}^{N}\partial_{k}f_{k}(u)\circ d\b_{t}^{k}
\]
and well-posedness to such SPDE is proven by a kinetic approach. This corresponds to \eqref{eq:full_smooth_noise-intro} with nonlinear, spatially homogeneous $\L_{k}(x,r,p)=f'_{k}(r)p_{k}$. We emphasize that for the results obtained in \cite{LPS13} it is crucial that the random flux $\L=(\L_{k})_{k=1}^{N}$ is spatially homogeneous (i.e. does not depend on $x$), which would correspond to $H=(H_{k})_{k=1}^{N}$ being a constant matrix in our framework \eqref{eq:full_smooth_noise-intro}. Very recently, in the case of one driving Brownian motion, i.e.
\begin{equation}
du=\sum_{k=1}^{N}\partial_{k}f_{k}(x,u)\circ d\b_{t},\label{eq:SSCL_single_BM}
\end{equation}
where $\b$ is a real-valued Brownian motion, a generalization of the results from \cite{LPS13} to the spatially dependent case has been obtained in \cite{LPS14}. Due to the restriction to one-dimensional noise no rough paths techniques are required to handle \eqref{eq:SSCL_single_BM}.

\subsection{Notation}

We will now very briefly recall the elements of rough paths theory used in this paper. For more details we refer to \cite{FV10}. Let $T^{N}(\R^{d})=\R\oplus\R^{d}\oplus(\R^{d}\otimes\R^{d})\oplus\ldots\oplus(\R^{d})^{\otimes N}$ be the truncated step-$N$ tensor algebra. For paths in $T^{N}(\R^{d})$ starting at the fixed point $e:=1+0+\ldots+0$, one may define $\beta$-Hölder and $p$-variation metrics, extending the usual metrics for paths in $\R^{d}$ starting at zero: The \textit{homogeneous} $\beta$-Hölder and $p$-variation metrics will be denoted by $\ensuremath{d_{\beta-\text{Höl}}}$
 resp. $d_{p-\text{var}}$, the \textit{inhomogeneous} ones by $\ensuremath{\rho_{\beta-\text{Höl}}}$ resp. $\rho_{p-\text{var}}$ respectively. Note that both $\beta$-Hölder and $p$-variation metrics induce the same topology on the path spaces. Corresponding norms are defined by 
$\ensuremath{\|\cdot\|_{\beta-\text{Höl}}=d_{\beta-\text{Höl}}(\cdot,0)}$ and $\|\cdot\|_{p-\text{var}}=d_{p-\text{var}}(\cdot,0)$ where $0$ denotes the constant $e$-valued path.

A geometric $\beta$-Hölder rough path $\mathbf{x}$ is a path in $T^{\lfloor1/\beta\rfloor}(\R^{d})$ which can be approximated by lifts of smooth paths in the $\ensuremath{d_{\beta-\text{Höl}}}$ metric; geometric $p$-rough paths are defined similarly. Given a rough path $\mathbf{x}$, the projection on the first level is an $\R^{d}$-valued path and will be denoted by $\pi_{1}(\mathbf{x})$. It can be seen that rough paths actually take values in the smaller set $G^{N}(\R^{d})\subset T^{N}(\R^{d})$, where $G^{N}(\R^{d})$ denotes the free step-$N$ nilpotent Lie group with $d$ generators. The Carnot-Caratheodory metric turns $(G^{N}(\R^{d}),d)$ into a metric space. Consequently, we denote by 
\begin{align*} C^{0,\beta-\text{H\"ol}}_0(I,G^{\lfloor 1/\beta \rfloor}(\R^d))\quad \text{ and} \quad C^{0,p-\text{var}}_0(I,G^{\lfloor p \rfloor}(\R^d)) \end{align*} the rough paths spaces where $\beta\in(0,1]$ and $p\in[1,\infty)$. Note that both spaces are Polish spaces.

\section{Definitions and Notation\label{sec:def_not}}

For a matrix $A=(a_{i,j})_{i,j=1,...,d}$ we write $A_{i}^{j}=a_{i,j}$, $A^{j}=(a_{i,j})_{i=1,...,d}$ and $A_{i}=(a_{i,j})_{j=1,...,d}^{t}$. Let $H=(H^{1},...,H^{N_{1}})$ be a collection of $C^{1}(\R^{d};\R^{d})$ vector fields. We define
\[
\div H:=(\div H^{1},...,\div H^{N_{1}})
\]
and assume  $\div H=0$. In the following we let $\Div$ denote the total divergence, i.e. for a vector-valued function $f=f(x,u)\in C^{1}(\R^{d}\times\R)$ and for $u\in C^{1}(\R^{d})$ we set 
\[
\Div f(x,u)=(\div f)(x,u)+(\partial_{u}f)(x,u)\cdot\nabla u,
\]
while $\div f(x,u)=\sum_{k=1}^{d}(\partial_{x_{k}}f^{k})(x,u)$. Moreover, we let $\nabla f$ denote the partial gradient, that is $\nabla f(x,u)=\left((\partial_{x_{i}}f^{j})(x,u)\right)_{i,j,=1,\dots,d}$. For all $R,M>0$, $t\in[0,T]$, $x_{0}\in\R^{d}$ we define time-space cones by
\[
K_{R,M}(t,x_{0}):=\{(r,x)|\ x\in B_{R+M(t-r)}(x_{0})\}.
\]
We let $C^{k}(\R^{d})$ be the usual spaces of $k$-times continuously differentiable functions on $\R^{d}$ and let $C_{b}^{k}(\R^{d})$ denote the subset of bounded functions. Analogously, we define $\Lip_{b}^{\g}$ to be the bounded $\g$-Lipschitz continuous functions.

\subsection{Definition of a weak entropy solution}

The replacement of Brownian motion in \eqref{eq:full_smooth_noise-intro} by a continuously differentiable path $z$ leads us to the study of the following evolution equation
\begin{align}
\partial_{t}u+\Div f(t,x,u) & =F(t,x,u)+\nabla u\cdot H(x)\dot{z}_{t}^{1}+u\nu(t,x)\dot{z}_{t}^{2}+g(t,x)\dot{z}_{t}^{3}\label{eqn:full_noise}\\
u(0) & =u_{0}\in L^{\infty}(\R^{d})\nonumber 
\end{align}
on $[0,T]\times\R^{d}$ with $f,F$  continuous, $d,N_{1},N_{2},N_{3}\in\N$,
\begin{align*}
z=(z^{1},z^{2},z^{3})\in & C^{1}([0,T];\R^{N_{1}+N_{2}+N_{3}}),\\
H\in & (C_{b}^{2}\cap\Lip)(\R^{d};\R^{d\times N_{1}}),\\
\nu\in & C^{0}([0,T];(C_{b}^{2}\cap\Lip)(\R^{d};\R^{N_{2}})),\\
g\in & C^{0}([0,T];(C_{b}^{2}\cap\Lip)(\R^{d};\R^{N_{3}}))
\end{align*}
 and assuming $\div(H)=0$. Since (informally)
\[
\nabla u\cdot H\dot{z}_{t}^{1}=\Div(uH\dot{z}_{t}^{1}),
\]
we may rewrite \eqref{eqn:full_noise} as
\begin{equation}
\partial_{t}u+\Div\td f(t,x,u)=\td F(t,x,u)\label{eqn:stoch_3-2}
\end{equation}
with%
{} 
\begin{align}
\td f(t,x,u) & =f(t,x,u)-uH(x)\dot{z}_{t}^{1}\label{eq:f_td}\\
\td F(t,x,u) & =F(t,x,u)+u\nu(t,x)\dot{z}_{t}^{2}+g(t,x)\dot{z}{}_{t}^{3}.\nonumber 
\end{align}
Thus, \eqref{eqn:full_noise} may be rewritten in terms of an inhomogeneous scalar conservation law
\begin{align}
\partial_{t}u+\Div f(t,x,u) & =F(t,x,u)\label{eq:inhomogen_SCL}\\
u(0) & =u_{0}\in L^{\infty}(\R^{d})\nonumber 
\end{align}
for which the well-developed deterministic theory of entropy solutions and their stability may be applied, provided $z\in C^{1}([0,T];\R^{N_{1}+N_{2}+N_{3}})$. The removal of this regularity assumption on the driving signal $z$ is the main point of this paper.
\begin{defn}
\label{def:entropy_for_full_noise}We call $u\in L^{\infty}([0,T]\times\R^{d};\R)$ a weak entropy solution to \eqref{eq:inhomogen_SCL}  if
\begin{enumerate}
\item For all $k\in\R$, $\vp\in C_{c}^{\infty}((0,T)\times\R^{d};\R_{+})$ 
\begin{align*}
\int_{0}^{T}\int_{\R^{d}} & |u-k|\partial_{t}\vp+\sgn(u-k)(f(t,x,u)-f(t,x,k))\nabla\vp\\
 & +\sgn(u-k)(F(t,x,u)-\div f(t,x,k))\vp dxdt\ge0.
\end{align*}

\item There exists a zero set $\mcE\subseteq[0,T]$ such that for $t\in[0,T]\setminus\mcE$ the function $u(t,x)$ is defined for a.e. $x\in\R^{d}$ and for all $r>0$ 
\[
\lim_{t\to0,t\in\R_{+}\setminus\mcE}\int_{B_{r}(0)}|u(t,x)-u_{0}(x)|dx=0.
\]

\end{enumerate}
Moreover, a function $u$ is said to be a weak entropy solution to \eqref{eqn:full_noise} if $u$ is a weak entropy solution to \eqref{eqn:stoch_3-2}
\end{defn}
As concerning the well-posedness of \eqref{eqn:full_noise} we will work with the following set of assumptions

\begin{hypothesis}\label{hyp:H0-H2}
\begin{enumerate}
\item [($H1$)] $f,F$ are continuous, $\partial_{u}f,\partial_{u}\nabla f,\nabla^{2}f,\partial_{u}F,\nabla F$ exist continuously and
\begin{align*}
\partial_{u}f & \in L^{\infty}([0,T]\times\R^{d}\times[-U,U]),\\
F-\div f,\partial_{u}(F-\div f) & \in L^{\infty}([0,T]\times\R^{d}\times[-U,U]),
\end{align*}
 for all $U,T>0$.
\item [($H2$)] For all $U,T>0$: $\nabla\partial_{u}f\in L^{\infty}([0,T]\times\R^{d}\times[-U,U])$, $\partial_{u}F\in L^{\infty}([0,T]\times\R^{d}\times[-U,U])$ and 
\[
\int_{0}^{T}\int_{\R^{d}}\|\nabla(F-\div f)(t,x,\cdot)\|_{L^{\infty}([-U,U])}dxdt<\infty.
\]

\item [($H2^*$)] For all $U,R,T>0$: $\nabla\partial_{u}f\in L^{\infty}([0,T]\times\R^{d}\times[-U,U])$, $\partial_{u}F\in L^{\infty}([0,T]\times\R^{d}\times[-U,U])$ and  
\[
\int_{0}^{T}\int_{B_{R}(0)}\|\nabla(F-\div f)(t,x,\cdot)\|_{L^{\infty}([-U,U])}dxdt<\infty.
\]

\item [($H3$)] $(\div f-F)(\cdot,\cdot,0)\in L^{\infty}([0,T]\times\R^{d})$ and $\partial_{u}(\div f-F)\in L^{\infty}([0,T]\times\R^{d}\times\R)$.
\end{enumerate}
\end{hypothesis}

We recall 
\begin{defn}
Let $u\in L_{loc}^{1}(\R^{d})$. Define
\begin{align*}
\TV(u) & =\sup{\left\{ \int_{\R^{d}}u\div\psi dx|\ \psi\in C_{c}^{1}(\R^{d};\R^{d})\text{ and }\|\psi\|_{\infty}\le1\right\} }\\
BV(\R^{d}) & =\{u\in L_{loc}^{1}(\R^{d})|\ \TV(u)<\infty\}.
\end{align*}

\end{defn}
From \cite{K70,LM11} and Appendix \ref{sec:app_det_SCL} we obtain
\begin{prop}
\label{prop:general well posedness}Let $u_{0}\in L^{\infty}(\R^{d})$.
\begin{enumerate}
\item Suppose that $f$, $F$ satisfy $(H1),(H2)$ and $u_{0}\in(L^{\infty}\cap L^{1}\cap BV)(\R^{d})$.  Then weak entropy solutions to \eqref{eq:inhomogen_SCL} are unique. 
\item Suppose that $f$, $F$ satisfy $(H1),(H3)$. Then there exists a weak entropy solution $u$ to \eqref{eq:inhomogen_SCL}. Moreover, $u$ may be chosen such that $t\mapsto u(t)$ is right-continuous in $L_{loc}^{1}(\R^{d})$.
\item Suppose that $f$, $F$ satisfy $(H1)$, $(H2^{*})$, $(H3)$ and $u_{0}\in(L^{\infty}\cap L^{1}\cap BV)(\R^{d})$. Then there exists a unique weak entropy solution to \eqref{eq:inhomogen_SCL}.
\end{enumerate}
\end{prop}
\begin{proof}
(i): Follows from \cite[Theorem 2.5]{LM11}. (ii): Proven in \cite{K70}. (iii): Follows from Theorem \ref{thm:loc_entropy_stability} in Appendix \ref{sec:app_det_SCL} below.
\end{proof}
For simplicity we will assume weak entropy solutions to be right-continuous in $L_{loc}^{1}(\R^{d})$. Due to Proposition \ref{prop:general well posedness} (ii) this does not restrict the applicability of our results.

Note that $(H2)$ for $f$, $F$ does not imply $(H2)$ for $\td f$, $\td F$ defined in \eqref{eq:f_td} while this is the case for $(H2^{*})$. In order to have well-posedness for \eqref{eqn:full_noise} it is thus important to work with the localized condition $(H2^{*})$ instead, as in Proposition \ref{prop:general well posedness} (ii).

\section{Transformation for smooth noise\label{sec:transformation}}

In this section we consider
\begin{align}
\partial_{t}u+\Div f(t,x,u) & =F(t,x,u)+\nabla u\cdot H(x)\dot{z}_{t}^{1}+u\nu(x)\dot{z}_{t}^{2}+g(x)\dot{z}_{t}^{3},\label{eq:full_smooth_noise}\\
u(0) & =u_{0},\nonumber 
\end{align}
on $[0,T]\times\R^{d}$ with $d,N_{1},N_{2},N_{3}\in\N$, $f,F$ satisfying $(H1)$, $(H2^{*})$, $(H3)$, 
\begin{align*}
z=(z^{1},z^{2},z^{3})\in & C^{1}([0,T];\R^{N_{1}+N_{2}+N_{3}}),\\
H\in & (C_{b}^{3}\cap\Lip)(\R^{d};\R^{d\times N_{1}}),\\
\nu,g\in & (C_{b}^{2}\cap\Lip)(\R^{d}),
\end{align*}
and $\div(H)=0$.

We emphasize that Proposition \ref{prop:general well posedness} fails when $z=(z^{1},z^{2},z^{3})$ ceases to be $C^{1}([0,T];\R^{N_{1}+N_{2}+N_{3}})$. In particular, the case of $z$ being Brownian motion is not covered. In the following we will show how to transform \eqref{eq:full_smooth_noise} into a scalar conservation law in ``robust'' form, which will in turn allow the development of a rough pathwise theory for \eqref{eq:full_smooth_noise}. The point is to find a view on \eqref{eq:full_smooth_noise} which (to the extend possible) does not involve derivatives of the driving noise $z$. 

In order to do so, we split the presentation into two parts, first dealing with pure transport noise $\nabla u\cdot H\dot{z}_{t}^{1}$ then with affine-linear noise $u\nu(x)\dot{z}_{t}^{2}+g(x)\dot{z}_{t}^{3}$. Finally, in Section \ref{sub:full_transf} below, both of these transformations will be applied to \eqref{eq:full_smooth_noise} to yield its robust form.

\subsection{Transport noise}

In this section we consider
\begin{equation}
\partial_{t}u+\Div f(t,x,u)=F(t,x,u)+\nabla u\cdot H(x)\dot{z}_{t}^{1},\label{eq:linear_mult_smooth}
\end{equation}
on $\R^{d}$ with $z^{1}\in C^{1}([0,T];\R^{N_{1}})$, $H\in(C_{b}^{2}\cap\Lip)(\R^{d};\R^{d\times N_{1}})$, $\div(H)=0$ and $f,F$ satisfying $(H1)$. Let $\psi$ be the flow of $C^{2}$-diffeomorphisms induced by 
\begin{align*}
\dot{\psi}_{t}(x) & =-H(\psi_{t}(x))\dot{z}_{t}^{1}\\
\psi_{0}(x) & =x.
\end{align*}
Note that $\psi_{t}$ is volume preserving, since $\div(H)=0$. We aim to transform \eqref{eq:linear_mult_smooth} into its ``robust'' form by setting $v(t,x)=u(t,\psi_{t}(x))$. In the context of viscosity solutions an analogous transformation has been studied for example in \cite{CFO11,LS98,FO13}. An informal computation reveals
\begin{equation}
\begin{split}\partial_{t}v(t,x)= & (\partial_{t}u)(t,\psi_{t}(x))+(\nabla u)(t,\psi_{t}(x))\cdot\partial_{t}\psi_{t}(x)\\
= & (-\Div f(t,x,u))(t,\psi_{t}(x))+F(t,x,u)(t,\psi_{t}(x))\\
 & +(\nabla u\cdot H)(t,\psi_{t}(x))\dot{z}_{t}^{1}+(\nabla u)(t,\psi_{t}(x))\cdot\partial_{t}\psi_{t}(x)\\
= & (-\Div f(t,x,u))(t,\psi_{t}(x))+F(t,\psi_{t}(x),v).
\end{split}
\end{equation}
 By Proposition \ref{prop:div_transf} at least for $u\in C^{1}(\R^{d})$ we have
\[
\left(\Div f(t,x,u)\right)(t,\psi_{t}(x))=\Div\left((D\psi_{t})^{-1}f(t,\psi_{t},u(t,\psi_{t}))\right)(t,x).
\]
 Hence,
\begin{equation}
\partial_{t}v(t,x)+\Div f^{\psi}(t,x,v)=F^{\psi}(t,x,v),\label{eqn:transf}
\end{equation}
 with
\begin{align*}
f^{\psi}(t,x,v) & =({D\psi_{t}^{-1}})_{|\psi_{t}(x)}f(t,\psi_{t}(x),v),\\
F^{\psi}(t,x,v) & =F(t,\psi_{t}(x),v).
\end{align*}
This informal calculation may be made rigorous
\begin{prop}
\label{prop:inner_transform}A function $u$ is a weak entropy solution to
\begin{equation}
\partial_{t}u+\Div f(t,x,u)=F(t,x,u)+\nabla u\cdot H(x)\dot{z}_{t}^{1},\label{eq:transport_1}
\end{equation}
iff $v(t,x)=u(t,\psi_{t}(x))$ is a weak entropy solution to 
\begin{equation}
\partial_{t}v(t,x)+\Div f^{\psi}(t,x,v)=F^{\psi}(t,x,v)\label{eq:transport_transformed}
\end{equation}
where
\begin{align*}
f^{\psi}(t,x,v) & =({D\psi_{t}^{-1}})_{|\psi_{t}(x)}f(t,\psi_{t}(x),v),\\
F^{\psi}(t,x,v) & =F(t,\psi_{t}(x),v).
\end{align*}
\end{prop}
\begin{proof}
Assume that $u\in L^{\infty}([0,T]\times\R^{d})$ is a weak entropy solution to \eqref{eq:transport_1} (in the sense of Definition \ref{def:entropy_for_full_noise}). Hence,
\begin{align*}
\int_{0}^{T}\int_{\R^{d}}|u-k|\partial_{t}\vp & +\sgn(u-k)(f(t,x,u)-f(t,x,k))\cdot\nabla\vp-|u-k|H\dot{z}^{1}\cdot\nabla\vp\\
 & +\sgn(u-k)(F(t,x,u)-\div f(t,x,k))\vp dxdt\ge0,
\end{align*}
 for all $k\in\R$, $\vp\in C_{c}^{\infty}([0,T]\times\R^{d})$. %
Substituting $\psi_{t}^{-1}(x)$ yields
\begin{align*}
\int_{0}^{T}\int_{\R^{d}} & |v-k|(\partial_{t}\vp)(t,\psi_{t}(x))\\
 & +\sgn(v-k)(f(t,\psi_{t}(x),v)-f(t,\psi_{t}(x),k))\cdot(\nabla\vp)(t,\psi_{t}(x))\\
 & -|v-k|H(\psi_{t}(x))\dot{z}^{1}\cdot(\nabla\vp)(t,\psi_{t}(x))\\
 & +\sgn(v-k)(F(t,\psi_{t}(x),v)-(\div f)(t,\psi_{t}(x),k))\vp(t,\psi_{t}(x))dxdt\ge0,
\end{align*}
 where we use that $\psi_{t}$ is volume preserving. We note
\begin{align*}
\partial_{t}[\vp(t,\psi_{t}(x))] & =(\partial_{t}\vp)(t,\psi_{t}(x))+(\nabla\vp)(t,\psi_{t}(x))\cdot\partial_{t}\psi_{t}(x)\\
 & =(\partial_{t}\vp)(t,\psi_{t}(x))-(\nabla\vp)(t,\psi_{t}(x))\cdot H(\psi_{t}(x))\dot{z}_{t}^{1}
\end{align*}
and
\begin{align*}
(\nabla\vp)(t,\psi_{t}(x)) & =\nabla(\vp(t,\psi_{t}(x))\cdot{D\psi_{t}^{-1}}_{|\psi_{t}(x)}\\
 & ={(D\psi_{t}^{-1})^{t}}_{|\psi_{t}(x)}\nabla(\vp(t,\psi_{t}(x)).
\end{align*}
By Proposition \ref{prop:div_transf} we have
\[
(\div f)(t,\psi_{t}(x),k)=\div\left((D\psi^{-1})_{|\psi_{t}(x)}f(t,\psi_{t}(x),k)\right).
\]
Hence,
\begin{align*}
\int_{0}^{T}\int_{\R^{d}} & |v-k|\partial_{t}(\vp(t,\psi_{t}))\\
 & +\sgn(v-k)(f(t,\psi_{t},v)-f(t,\psi_{t},k))\cdot{(D\psi_{t}^{-1})^{t}}_{|\psi_{t}}\nabla(\vp(t,\psi_{t})\\
 & +\sgn(v-k)(F(t,\psi_{t},v)-\div((D\psi^{-1})_{|\psi_{t}}f(t,\psi_{t},k)))\vp(t,\psi_{t})dxdt\ge0,
\end{align*}
for all $k\in\R$, $\vp\in C_{c}^{\infty}([0,T]\times\R^{d})$. This is equivalent to
\begin{align*}
\int_{0}^{T}\int_{\R^{d}} & |v-k|\partial_{t}\vp\\
 & +\sgn(v-k)({(D\psi_{t}^{-1})}_{|\psi_{t}}f(t,\psi_{t},v)-{(D\psi_{t}^{-1})}_{|\psi_{t}}f(t,\psi_{t},k))\nabla\vp\\
 & +\sgn(v-k)(F(t,\psi_{t},v)-\div((D\psi^{-1})_{|\psi_{t}}f(t,\psi_{t},k))\vp dxdt\ge0,
\end{align*}
for all $k\in\R$, $\vp\in C_{c}^{\infty}(\R_{+}\times\R^{d})$. Hence, $v$ is a weak entropy solution to \eqref{eq:transport_transformed}. Following the above calculations in reverse order yields that $u$ is a weak entropy solution if $v$ is.\end{proof}
\begin{rem}
 
\begin{enumerate}
\item Another way to rigorously justify the informal calculations leading to \eqref{eqn:transf} would be to argue via a vanishing viscosity approximation, i.e. first approximate \eqref{eq:linear_mult_smooth} by
\[
\partial_{t}u^{\ve}+\Div f(t,x,u^{\ve})=\ve\D u^{\ve}+F(t,x,u^{\ve})+\nabla u^{\ve}\cdot H(x)\dot{z}_{t}^{1}
\]
then compute the transformed equation by classical calculus and take $\ve\to0$. In order to guarantee that $u^{\ve}$ indeed converges to the (unique) weak entropy solution $u$ more restrictive assumptions on $f,F$ would be necessary. 
\item We emphasize that Proposition \ref{prop:inner_transform} does not yield any claim on the existence and uniqueness of the concerned weak entropy solutions. Again, more restrictive assumptions on $f,F$ would be necessary. 
\end{enumerate}
\end{rem}

\subsection{Affine linear space-time noise}

We consider 
\begin{align}
\partial_{t}u+\Div f(t,x,u) & =F(t,x,u)+u\nu(t,x)\dot{z}_{t}^{2}+g(t,x)\dot{z}_{t}^{3}\label{eq:smooth_affine_linear}\\
u(0,x) & =u_{0}(x)\nonumber 
\end{align}
on $\R^{d}$ with $N_{2},N_{3}\in\N$,
\begin{align*}
(z^{2},z^{3})\in & C^{1}([0,T];\R^{N_{2}+N_{3}}),\\
\nu,g\in & C^{0}([0,T];(C_{b}^{2}\cap\Lip)(\R^{d})),
\end{align*}
 and $f,F$ satisfying $(H1)$, $(H2^{*})$, $(H3)$. It is then easy to see that also $f$ and
\[
\td F(t,x,u):=F(t,x,u)+u\nu(t,x)\dot{z}_{t}^{2}+g(t,x)\dot{z}_{t}^{3}
\]
satisfy $(H1),(H2^{*}),(H3)$ and thus there is a unique weak entropy solution $u$ to \eqref{eq:smooth_affine_linear} by Proposition \ref{prop:general well posedness}.
\begin{rem}
We note that $(H2)$ for $f,F$ does not necessarily imply $(H2)$ for $f,\td F$ as defined in \eqref{eq:f_td} since 
\[
\nabla\td F=\nabla F+u\nabla\nu\dot{z}^{2}+\nabla g\dot{z}^{3}
\]
is not known to be in $L^{1}([0,T]\times\R^{d})$. The localization of $(H2)$ in form of $(H2^{*})$ thus becomes crucial at this point.
\end{rem}
Let $\phi$ be the flow of $C^{2}$-diffeomorphisms corresponding to 
\begin{align*}
\dot{\phi}(t,x) & =\phi(t,x)\nu(t,x)\dot{z}_{t}^{2}\\
\phi(0,x) & =Id_{\R},
\end{align*}
i.e. $\phi(t,x)r=re^{\int_{0}^{t}\nu(\tau,x)\dot{z}_{\tau}^{2}d\tau}$ . For notational convenience we set 
\[
\mu(t,x):=-\int_{0}^{t}\nu(r,x)\dot{z}_{r}^{2}dr.
\]
Moreover, let $\vr$ be the flow of $C^{2}$-diffeomorphisms to
\[
\dot{\vr}(t,x)=e^{\mu(t,x)}g(t,x)\dot{z}_{t}^{3},
\]
i.e. $\vr(t,x)=\int_{0}^{t}e^{\mu(r,x)}g(r,x)\dot{z}_{r}^{3}dr$. 
\begin{prop}
\label{prop:affine_linear_transf}Let $u_{0}\in(L^{\infty}\cap L^{1}\cap BV)(\R^{d})$. A function $u$ is the unique weak entropy solution to 
\begin{align}
\partial_{t}u+\Div f(t,x,u) & =F(t,x,u)+u\nu(t,x)\dot{z}_{t}^{2}+g(t,x)\dot{z}_{t}^{3}\label{eq:linear_noise-2}\\
u(0) & =u_{0}\nonumber 
\end{align}
 iff $v(t,x)=e^{\mu(t,x)}u(t,x)-\vr(t,x)$ is the unique weak entropy solution%
\footnote{We note that $^{\phi}f,{}^{\phi}F$ do not necessarily satisfy $(H2)$ nor $(H3)$ anymore. Existence and uniqueness of a weak entropy solution to \eqref{eq:linear noise transformed-1} is part of the proof. %
} to
\begin{align}
\partial_{t}v & +\Div_{\vr}^{\phi}f(t,x,v)={}_{\vr}^{\phi}F(t,x,v),\label{eq:linear noise transformed-1}\\
v(0) & =u_{0}\nonumber 
\end{align}
where
\begin{align*}
_{\vr}^{\phi}f(t,x,v):= & e^{\mu(t,x)}f(t,x,e^{-\mu(t,x)}(v+\vr(t,x)))\\
_{\vr}^{\phi}F(t,x,v):= & e^{\mu(t,x)}F(t,x,e^{-\mu(t,x)}(v+\vr(t,x)))\\
 & +f(t,x,e^{-\mu(t,x)}(v+\vr(t,x)))\nabla e^{\mu(t,x)}.
\end{align*}
\end{prop}
\begin{proof}
For this so-called ``outer transformation'' (cf. \cite{FO13}) it seems more convenient to argue via a vanishing viscosity approximation than to work with the entropy formulation directly as it was done in Proposition \ref{prop:inner_transform}. In order to obtain the existence and uniqueness of a weak entropy solution to \eqref{eq:linear noise transformed-1} we shall first consider an approximation via localization of $ $$f,F,\nu,g$. As a second step we consider smooth approximations of these localizations. We then consider vanishing viscosity approximations which allow to calculate the transformation explicitly. We may then recover the general cases by stability of solutions to scalar conservation laws.

\textbf{Step 1:} Smooth, compactly supported data

We start with the case of smooth, compactly supported data, i.e. assume in addition $f,F,\nu,g,z,u_{0}$ to be smooth with 
\begin{equation}
f(t,x,u)=F(t,x,u)=\nu(t,x)=g(t,x)=0,\quad\forall|x|\ge m,(t,u)\in[0,T]\times\R,\label{eq:compact_supp-1}
\end{equation}
for some $m>0$. In particular, $f,F$ satisfy $(H1),(H2),(H3)$. We then consider a vanishing viscosity approximation, i.e. 
\begin{align}
\partial_{t}u^{\ve}+\Div f(t,x,u^{\ve}) & =\ve\D u^{\ve}+F(t,x,u^{\ve})+u^{\ve}\nu(t,x)\dot{z}_{t}^{2}+g(t,x)\dot{z}_{t}^{3}\label{eq:linear_noise-1-1}\\
u^{\ve}(0) & =u_{0}.\nonumber 
\end{align}
The existence of a unique classical solution to \eqref{eq:linear_noise-1-1} follows from standard theory (cf. e.g. \cite{LSU67}) and from \cite[Theorem 4]{K70} we know that 
\begin{equation}
u^{\ve}\to u\quad\text{in }L^{1}([0,T];L_{loc}^{1}(\R^{d}))\label{eq:u_conv-1}
\end{equation}
 and $dt\otimes d\xi$ almost everywhere (selecting subsequences if necessary). Due to $(H3)$ and the maximum principle (cf. also Lemma \ref{lem:linfty_bound} below) $u^{\ve}$ is uniformly bounded in $L^{\infty}([0,T]\times\R^{d})$. Setting
\[
v^{\ve,1}(t,x)=e^{\mu(t,x)}u^{\ve}(t,x)
\]
we obtain
\begin{align*}
\partial_{t}v^{\ve,1}= & e^{\mu(t,x)}\partial_{t}u^{\ve}(t,x)-\nu(t,x)\dot{z}_{t}^{2}e^{\mu(t,x)}u^{\ve}(t,x)\\
= & e^{\mu(t,x)}(\ve\D u^{\ve}-\Div f(t,x,u^{\ve})+F(t,x,u^{\ve})+u^{\ve}(t,x)\nu(t,x)\dot{z}_{t}^{2}+g(t,x)\dot{z}_{t}^{3})\\
 & -\nu(t,x)\dot{z}_{t}v(t,x)\\
= & e^{\mu(t,x)}\ve\D e^{-\mu(t,x)}v^{\ve,1}-e^{\mu(t,x)}\Div f(t,x,e^{-\mu(t,x)}v^{\ve,1})\\
 & +e^{\mu(t,x)}F(t,x,e^{-\mu(t,x)}v^{\ve,1})+e^{\mu(t,x)}g(t,x)\dot{z}_{t}^{2}.
\end{align*}
We now set $\vr(t,x)=\int_{0}^{t}e^{\mu(r,x)}g(r,x)\dot{z}_{r}^{3}dr$ and 
\[
v^{\ve}(t,x)=v^{\ve,1}(t,x)-\vr(t,x).
\]
Then
\begin{align*}
\partial_{t}v^{\ve}= & \partial_{t}v^{\ve,1}-e^{\mu(t,x)}g(t,x)\dot{z}_{t}^{3}\\
= & e^{\mu(t,x)}\ve\D e^{-\mu(t,x)}(v^{\ve}+\vr)-e^{\mu(t,x)}\Div f(t,x,e^{-\mu(t,x)}(v^{\ve}+\vr))\\
 & +e^{\mu(t,x)}F(t,x,e^{-\mu(t,x)}(v^{\ve}+\vr)).
\end{align*}
Since
\[
e^{\mu(t,x)}\Div f(t,x,u)=\Div(e^{\mu(t,x)}f(t,x,u))-f(t,x,u)\nabla e^{\mu(t,x)}
\]
we have
\begin{align*}
\partial_{t}v^{\ve} & +\Div_{\vr}^{\phi}f(t,x,v^{\ve})=\ve_{\vr}^{\phi}Lv^{\ve}+{}_{\vr}^{\phi}F(t,x,v^{\ve}),
\end{align*}
where the linear, strongly elliptic operator $_{\vr}^{\phi}L:H^{2}(\mcO)\cap H_{0}^{1}(\mcO)\to L^{2}(\mcO)$ is defined by
\[
_{\vr}^{\phi}Lv:=e^{\mu}\D e^{-\mu}(v+\vr)=\D v-2\nabla\mu\cdot\nabla v+v(|\nabla\mu|^{2}-\D\mu)+e^{\mu}\D(e^{-\mu}\vr).
\]
Due to \eqref{eq:u_conv-1} we have
\[
v^{\ve}\to v:=e^{\mu}u\quad\text{in }L^{1}([0,T];L_{loc}^{1}(\R^{d}))
\]
which is easily seen to imply that $v$ is a weak entropy solution to \eqref{eq:linear noise transformed-1}. 

\textbf{Step 2:} $u_{0}\in(L^{\infty}\cap L^{1}\cap BV)(\R^{d})$ and $f,F,\nu,g$ having compact support in $x$, i.e. satisfy \eqref{eq:compact_supp-1}.

Let $u$ be the unique weak entropy solution to \eqref{eq:linear_noise-2}. We aim to remove the additional smoothness assumptions on the data required in step one. Let $f^{\d},F^{\d},\nu^{\d},g^{\d},z^{\d},u_{0}^{\d}$ be smooth approximations of $f,F,\nu,g,z,u_{0}$ respectively, obtained by mollification. Since $f,F$ satisfy $(H1),(H2),(H3)$ so do $f^{\d}$, $F^{\d}$. We have
\[
\left.\begin{gathered}\|u_{0}^{\d}-u_{0}\|_{L^{1}(\R^{d})}\\
\|\partial_{u}f^{\d}-\partial_{u}f\|_{L^{\infty}([0,T]\times\R^{d}\times[-U,U])}\\
\|F^{\d}-\div f^{\d}-(F-\div f)\|_{L^{\infty}([0,T]\times\R^{d}\times[-U,U])}\\
\|\nu^{\d}-\nu\|_{C^{0}([0,T]\times\R^{d})}\\
\|g^{\d}-g\|_{C^{0}([0,T]\times\R^{d})}\\
\|z^{\d}-z\|_{C^{1}([0,T])}
\end{gathered}
\right\} \to0,\quad\text{for }\d\to0
\]
for all $U,T>0$ and consider the sequence of unique weak entropy solutions $u^{\d}$ corresponding to 
\begin{align*}
\partial_{t}u^{\d}+\Div f^{\d}(t,x,u^{\d}) & =F^{\d}(t,x,u^{\d})+u^{\d}\nu^{\d}(t,x)\dot{z}_{t}^{\d,2}+g^{\d}(t,x)\dot{z}_{t}^{\d,3}\\
u^{\d}(0) & =u_{0}^{\d}.
\end{align*}
We note 
\begin{align*}
\td F(t,x,u)-\td F^{\d}(t,x,u)= & F(t,x,u)-F^{\d}(t,x,u)+u\nu\dot{z}^{2}-u\nu^{\d}\dot{z}_{t}^{\d,2}\\
 & +g(t,x)\dot{z}_{t}^{3}-g^{\d}(t,x)\dot{z}_{t}^{\d,3}.
\end{align*}
By step one we have that 
\[
v^{\d}(t,x):=e^{\mu^{\d}(t,x)}u^{\d}(t,x)-\vr^{\d}(t,x)
\]
is a weak entropy solution to 
\begin{align*}
\partial_{t}v^{\d} & +\Div_{\vr^{\d}}^{\phi^{\d}}f^{\d}(t,x,v^{\d})={}_{\vr^{\d}}^{\phi^{\d}}F^{\d}(t,x,v^{\d}).
\end{align*}
We note that $f^{\d}$, $F^{\d}$ satisfy $(H3)$ with uniform bounds. By Lemma \ref{lem:linfty_bound} this implies
\[
\mcV:=\|u\|_{L^{\infty}([0,T]\times\R^{d})}\vee\|u^{\d}\|_{L^{\infty}([0,T]\times\R^{d})}\le C<\infty.
\]
Due to Theorem \ref{thm:loc_entropy_stability} we have (with $M,\k_{0}^{*},\k^{*}$ defined as in Appendix \ref{sec:app_det_SCL}):
\[
\begin{split} & \sup_{t\in[0,T]}\int_{B_{R}(x_{0})}|u(t,x)-u^{\d}(t,x)|dx\\
 & \le e^{\kappa^{*}T}\int_{B_{R+MT}(x_{0})}|u_{0}(x)-u_{0}^{\d}(x)|dx\\
 & +Te^{(\kappa_{0}^{*}+\kappa^{*})T}\|\partial_{u}(f-f^{\d})\|_{L^{\infty}(K_{R,M}(T,x_{0})\times[-\mcV,\mcV])}\\
 & \times\Big(\TV(u_{0})+C\int_{0}^{T}\int_{\R^{d}}\|\nabla(\td F-\div f)(r,x,\cdot)\|_{L^{\infty}([-\mcV,\mcV])}dxdr\Big)\\
 & +e^{\kappa^{*}T}\int_{0}^{T}\int_{B_{R+MT}(x_{0})}\|((\td F-\td F^{\d})-\div(f-f^{\d}))(r,x\cdot)\|_{L^{\infty}([-\mcV,\mcV])}dxdr
\end{split}
\]
and thus
\[
\sup_{t\in[0,T]}\|u^{\d}(t)-u(t)\|_{L^{1}(K)}\to0
\]
for all compact sets $K\subseteq\R^{d}$. With $v:=e^{\mu}u-\vr$ we thus obtain
\[
\sup_{t\in[0,T]}\|v^{\d}(t)-v(t)\|_{L^{1}(K)}=\sup_{t\in[0,T]}\|e^{\mu^{\d}(t)}u^{\d}(t)-e^{\mu(t)}u(t)+\vr^{\d}(t)-\vr(t)\|_{L^{1}(K)}\to0,
\]
for all compact sets $K\subseteq\R^{d}$. It easily follows that $v$ is a weak entropy solution to \eqref{eq:linear noise transformed-1}. 

\textbf{Step 3:} $u_{0}\in(L^{\infty}\cap L^{1}\cap BV)(\R^{d})$

We argue as in the last step, approximating $f,F,\nu,g$ by localized approximations obtained by multiplication with a smooth cut-off function in the $x$-variable, i.e. set
\begin{align*}
f^{m}(t,x,u):= & \eta^{m}(x)f(t,x,u)\\
F^{m}(t,x,u):= & \eta^{m}(x)F(t,x,u)+\nabla\eta^{m}(x)\cdot f(t,x,u)\\
\nu^{m}(t,x):= & \eta^{m}(x)\nu(t,x)\\
g^{m}(t,x):= & \eta^{m}(x)g(t,x)
\end{align*}
where $\eta^{m}$ is a smooth function satisfying
\[
1_{B_{m}(0)}\le\eta^{m}\le1_{B_{m+1}(0)}.
\]
We note
\[
\td F^{m}-\div f^{m}=\eta^{m}(\td F-\div f)
\]
and thus $f^{m},\td F^{m}$ satisfy $(H1)$, $(H2^{*})$, $(H3)$. Let $u^{m}$ be the corresponding weak entropy solution. Since $f^{m},F^{m}$ satisfy $(H3)$ with uniform bounds we have 
\[
\mcV:=\|u^{m}\|_{L^{\infty}([0,T]\times\R^{d})}\le C<\infty,
\]
by Lemma \ref{lem:linfty_bound}. By Theorem \ref{thm:loc_entropy_stability} we obtain:
\[
\begin{split} & \sup_{t\in[0,T]}\int_{B_{R}(x_{0})}|u(t,x)-u^{m}(t,x)|dx\\
 & \le Te^{(\kappa_{0}^{*}+\kappa^{*})T}\|\partial_{u}(f-f^{m})\|_{L^{\infty}(K_{R,M}(T,x_{0})\times\R)}\\
 & \times\Big(\TV(u_{0})+C\int_{0}^{T}\int_{B_{R+MT}(x_{0})}\|\nabla(\td F^{m}-\div f^{m})(r,x,\cdot)\|_{L^{\infty}([-\mcV,\mcV])}dxdr\Big)\\
 & +e^{\kappa^{*}T}\int_{0}^{T}\int_{B_{R+MT}(x_{0})}\|((\td F-\td F^{m})-\div(f-f^{m}))(r,x\cdot)\|_{L^{\infty}([-\mcV,\mcV])}dxdr,
\end{split}
\]
 for all $R>0,x_{0}\in\R^{d}$. We observe
\[
\td F-\td F^{m}=(1-\eta^{m})(F+u\nu\dot{z}^{2}+g\dot{z}^{3}).
\]
Hence, for all $R>0,x_{0}\in\R^{d}$ and $m$ large enough we obtain
\begin{equation}
\begin{split}u^{m}\equiv & u,\quad\text{on }[0,T]\times B_{R}(x_{0}).\end{split}
\label{eq:convergence_in_m-1}
\end{equation}
Moreover, obviously
\begin{align*}
\mu^{m} & \equiv\mu\\
\vr^{m} & \equiv\vr,\quad\text{on }[0,T]\times B_{R}(x_{0}),
\end{align*}
for $m$ large enough. By step two,
\[
v^{m}=e^{\mu^{m}}u^{m}-\vr^{m}
\]
 are weak entropy solutions to \eqref{eq:linear noise transformed-1} with $_{\vr}^{\phi}f,{}_{\vr}^{\phi}F$ replaced by 
\begin{align}
_{\vr^{m}}^{\phi^{m}}f^{m}(t,x,v):= & e^{\mu^{m}(t,x)}f^{m}(t,x,e^{-\mu^{m}(t,x)}(v+\vr^{m}(t,x)))\nonumber \\
_{\vr^{m}}^{\phi^{m}}F^{m}(t,x,v):= & e^{\mu^{m}(t,x)}F^{m}(t,x,e^{-\mu^{m}(t,x)}(v+\vr^{m}(t,x)))\label{eq:fm}\\
 & +f^{m}(t,x,e^{-\mu^{m}(t,x)}(v+\vr^{m}(t,x)))\nabla e^{\mu^{m}(t,x)}.\nonumber 
\end{align}
Equation \eqref{eq:convergence_in_m-1} then implies that $v:=e^{\mu}u-\vr$ is a weak entropy solution to \eqref{eq:linear noise transformed-1}. 

\textbf{Step 4:} Uniqueness for \eqref{eq:linear noise transformed-1}

In step three we have obtained the existence of a weak entropy solution $v$ to \eqref{eq:linear noise transformed-1} as an $L^{1}([0,T];L_{loc}^{1}(\R^{d}))$ limit of weak entropy solutions $v^{m}$ corresponding to 
\begin{align*}
\partial_{t}v^{m} & +\Div{}_{\vr^{m}}^{\phi^{m}}f^{m}(t,x,v)={}_{\vr^{m}}^{\phi^{m}}F^{m}(t,x,v),\\
v(0) & =u_{0}
\end{align*}
where $_{\vr^{m}}^{\phi^{m}}f^{m},{}_{\vr^{m}}^{\phi^{m}}F^{m}$ are as in \eqref{eq:fm}. Note that since $u^{m}$ is uniformly bounded in $L^{\infty}([0,T]\times\R^{d})$ so is $v^{m}$. We observe that $_{\vr^{m}}^{\phi^{m}}f^{m},{}_{\vr^{m}}^{\phi^{m}}F^{m}$ have compact support in $x$ and 
\[
_{\vr^{m}}^{\phi^{m}}f^{m}(t,x,v)={}_{\vr}^{\phi}f(t,x,v)\quad\text{on }[0,T]\times B_{R}(0)\times\R
\]
for all $m>0$ large enough. Hence, uniqueness of weak entropy solutions to \eqref{eq:linear noise transformed-1} follows from Corollary \ref{cor:uniqueness_by_approx}. 
\end{proof}

\subsection{Full transformation\label{sub:full_transf}}

We now subsequently apply both of the transformations considered above. As before, let $d,N_{1},N_{2},N_{3}\in\N$, $f,F$ satisfying $(H1)$, $(H2^{*})$, $(H3)$,
\begin{align*}
z=(z^{1},z^{2},z^{3})\in & C^{1}([0,T];\R^{N_{1}+N_{2}+N_{3}}),\\
H\in & (C_{b}^{3}\cap\Lip)(\R^{d};\R^{d\times N_{1}}),\\
\nu,g\in & (C_{b}^{2}\cap\Lip)(\R^{d}),
\end{align*}
and assume $\div(H)=0$.

We define $\psi$ to be the flow of $C^{3}$-diffeomorphisms induced by
\begin{align*}
\dot{\psi}_{t} & =-H(\psi_{t})\dot{z}_{t}^{1}\\
\psi_{0} & =Id_{\R^{d}},
\end{align*}
and $\phi$ the one for 
\begin{align*}
\dot{\phi}_{t} & =\phi_{t}\nu(\psi_{t}(x))\dot{z}_{t}^{2}\\
\phi_{0}(x) & =Id_{\R}.
\end{align*}
Furthermore, we set $\vr(t,x):=\int_{0}^{t}{}^{\phi}g^{\psi}(r,x)\dot{z}_{r}^{3}dr,$ where 
\[
^{\phi}g^{\psi}(t,x):=\phi_{t}^{-1}(x)g(\psi_{t}(x))=e^{\mu(t,x)}g(\psi_{t}(x)),
\]
with $\mu(t,x):=-\int_{0}^{t}\nu(\psi_{r}(x))\dot{z}_{r}^{2}dr$. We obtain
\begin{prop}
\label{prop:full_transformation}Let $u_{0}\in(L^{\infty}\cap L^{1}\cap BV)(\R^{d})$. A function $u$ is the unique weak entropy solution to 
\begin{align}
\partial_{t}u+\Div f(t,x,u) & =F(t,x,u)+\nabla u\cdot H(x)\dot{z}_{t}^{1}+u\nu(x)\dot{z}_{t}^{2}+g(x)\dot{z}_{t}^{3},\label{eq:full_noise}\\
u(0) & =u_{0},\nonumber 
\end{align}
iff $v(t,x):=e^{\mu(t,x)}u(t,\psi_{t}(x))-\vr(t,x)$ is the unique weak entropy solution to 
\begin{align*}
\partial_{t}v+\Div{}_{\vr}^{\phi}f^{\psi}(t,x,v) & ={}_{\vr}^{\phi}F^{\psi}(t,x,v)\\
v(0) & =u_{0}
\end{align*}
with 
\begin{align}
_{\vr}^{\phi}f^{\psi}(t,x,v):= & e^{\mu(t,x)}{D\psi_{t}^{-1}}_{|\psi_{t}(x)}f(t,\psi_{t}(x),e^{-\mu(t,x)}(v+\vr(t,x)))\nonumber \\
_{\vr}^{\phi}F^{\psi}(t,x,v):= & e^{\mu(t,x)}F(t,\psi_{t}(x),e^{-\mu(t,x)}(v+\vr(t,x)))\nonumber \\
 & +{D\psi_{t}^{-1}}_{|\psi_{t}(x)}f(t,\psi_{t}(x),e^{-\mu(t,x)}(v+\vr(t,x)))\nabla e^{\mu(t,x)}\label{eq:full_transf}\\
= & e^{\mu(t,x)}F(t,\psi_{t}(x),e^{-\mu(t,x)}(v+\vr(t,x)))\nonumber \\
 & +{}_{\vr}^{\phi}f^{\psi}(t,x,v)\cdot\nabla\mu(t,x).\nonumber 
\end{align}
\end{prop}
\begin{proof}
We will successively apply both of the transformations introduced in the last sections. First we will deal with transport noise, then with affine-linear multiplicative noise. The crucial point is that along these transformations the equation remains in the class of inhomogeneous scalar conservation laws with source. 

We first note that there is a unique weak entropy solution $u$ to \eqref{eq:full_noise} since $\td f,\td F$ satisfy $(H1)$, $(H2^{*})$, $(H3)$. Let $v^{1}(t,x):=u(t,\psi_{t}(x)).$ Then, by Proposition \ref{prop:inner_transform}, $v^{1}$ is the unique weak entropy solution to
\[
\partial_{t}v^{1}+\Div f^{1}(t,x,v^{1})=v^{1}\nu(\psi_{t}(x))\dot{z}_{t}^{2}+g(\psi_{t}(x))\dot{z}_{t}^{3},
\]
with
\[
f^{1}(t,x,v):={D\psi_{t}^{-1}}_{|\psi_{t}(x)}f(t,\psi_{t}(x),v).
\]
We note, thanks to $\div H=0$, $\psi$ being the flow associated to $H$ and Proposition \ref{prop:div_transf} 
\[
\div f^{1}(t,x,v)=(\div f)(t,\psi_{t}(x),v)
\]
and thus $f^{1}$ and
\[
\td F^{1}(t,x,v):=F(t,\psi_{t}(x),v)+v\nu(\psi_{t}(x))\dot{z}_{t}^{2}+g(\psi_{t}(x))\dot{z}_{t}^{3}
\]
satisfy $(H1)$, $(H2^{*})$, $(H3)$. Now let $v(t,x)=e^{\mu(t,x)}v^{1}(t,x)-\vr(t,x)$. Then, by Proposition \ref{prop:affine_linear_transf}, $v$ is the unique weak entropy solution to 
\[
\partial_{t}v+\Div{}_{\vr}^{\phi}f^{\psi}(t,x,v)={}_{\vr}^{\phi}F^{\psi}(t,x,v)
\]
with $_{\vr}^{\phi}f^{\psi},{}_{\vr}^{\phi}F^{\psi}$ as in \eqref{eq:full_transf}. \textcolor{red}{}
\end{proof}

\section{Rough driving signals}

We now aim to give meaning to%
\footnote{For simplicity of the presentation we consider the case of $f$ being independent of $(t,x)$ and $F\equiv0$ in the following. The treatment of the general case, however, proceeds completely analogous.%
} 
\begin{equation}
du+\Div f(u)dt=\nabla u\cdot H(x)\circ d\mathbf{z}^{1}+u\nu\circ d\mathbf{z}^{2}+g(x)\circ d\mathbf{z}^{3}\label{eq:rough_full_noise}
\end{equation}
for \textbf{$\mathbf{z=(z^{1},z^{2},z^{3})}$} being a geometric $p$-rough path, recalling that the prototype of a (random) geometric $p$-rough path (with $p=2+\ve$) is given by Brownian motion plus its Lévy area. We will do so by considering smooth approximations $z^{n}$ of $\mathbf{z}$ in rough path metric and proving convergence of the associated approximants $u^{n}$ to a limit independent of the approximating sequence. We assume that there are $\g>p\ge1$, such that
\begin{align*}
f\in & C^{2}(\R),\\
H\in & \Lip_{b}^{\g+3}(\R^{d};\R^{N_{1}}),\nu\in\R^{N_{2}},g\in\Lip_{b}^{\g+2}(\R^{d};\R^{N_{3}}).
\end{align*}
Note that we now assume $\nu$ to be constant, which will be needed in order to establish a uniform $L^{\infty}$ bound for the approximants $u^{n}$ introduced above. Due to \cite{FV10} for any geometric $p$-rough path $\mathbf{z}\in C_{0}^{p-var}([0,T];G^{[p]}(\R^{d}))$ we may consider the flow of  diffeomorphisms
\begin{align}
d\psi_{t}^{\mathbf{z}}(x) & =-H(\mathbf{\psi}_{t}^{\mathbf{z}}(x))\circ d\mathbf{z}_{t}^{1},\quad\psi_{0}^{\mathbf{z}}(x)=x,\label{eq:flows_1}\\
d\mathbf{\phi}_{t}^{\mathbf{z}}(r) & =\mathbf{\phi}_{t}^{\mathbf{z}}(r)\nu\circ d\mathbf{z}_{t}^{2},\quad\phi_{0}^{\mathbf{z}}(r)=r,\nonumber 
\end{align}
i.e.
\[
\phi_{t}^{\mathbf{z}}(r)=re^{\nu(\mathbf{z}_{t}^{2}-\mathbf{z}_{0}^{2})}=:re^{-\mu_{t}^{\mathbf{z}}}
\]
and
\begin{equation}
\vr^{\mathbf{z}}(t,x)=\int_{0}^{t}e^{\mu_{r}^{\mathbf{z}}}g(\psi_{r}^{\mathbf{z}}(x))\circ d\mathbf{z}_{r}^{3}.\label{eq:flows_2}
\end{equation}
In order to obtain rough path stability of these diffeomorphisms we need to consider \eqref{eq:flows_1}, \eqref{eq:flows_2} ``simultaneously'' as a rough differential equation (RDE). Combining \cite[Lemma 13]{FO13} and \cite[Lemma 13]{CDFO13} we obtain%
\footnote{In fact, \cite[Lemma 13]{CDFO13} is formulated in the Hölder framework. It is, however, a simple exercise to see that an analogous result holds true also in the $p$-variation case.%
} 
\begin{lem}
\label{lem:hoelder_rp}Let $\g>p\ge1$. Assume
\begin{align*}
H & \in\Lip{}^{\g+3}(\R^{d},\R^{d\times N_{1}}),\ \nu\in\R^{N_{2}},\ g\in\Lip{}^{\g+2}(\R^{d},\R^{N_{3}}).
\end{align*}
Then for all $R>0$ there exist%
\footnote{The constants $C,K$ are non-decreasing in all arguments.%
} 
\begin{align*}
C & =C(R,\|H\|_{Lip^{\g+3}},|\nu|,\|g\|_{Lip^{\g+2}})\\
K & =K(R,\|H\|_{Lip^{\g+3}},|\nu|,\|g\|_{Lip^{\g+2}})
\end{align*}
such that for all geometric $p$-rough paths $\mathbf{y},\mathbf{z}\in C_{0}^{p-var}([0,T];G^{[p]}(\R^{d}))$ satisfying $\|\mathbf{y}\|_{p-var;[0,T]},\|\mathbf{z}\|_{p-var;[0,T]}\le R$ we have
\begin{align*}
\|D^{n}(\psi^{\mathbf{y}}-\psi^{\mathbf{z}})\|_{p-var;[0,T]} & \le C\rho_{p-var}(\mathbf{y},\mathbf{z})\\
\|D^{n}((\psi^{\mathbf{y}})^{-1}-(\psi^{\mathbf{z}})^{-1})\|_{p-var;[0,T]} & \le C\rho_{p-var}(\mathbf{y},\mathbf{z})
\end{align*}
for all $n\in\{0,1,2,3\}$ and 
\begin{align*}
\|D^{n}\psi^{\mathbf{y}}\|_{p-var;[0,T]} & \le K\\
\|D^{n}(\psi^{\mathbf{y}})^{-1}\|_{p-var;[0,T]} & \le K
\end{align*}
for all $n\in\{1,2,3\}$. Analogous properties for $\vr^{\mathbf{y}}$ (and trivially for $\phi^{\mathbf{y}}$) are satisfied.\end{lem}
\begin{thm}
\label{thm:RP_limit}Let $T\ge0$, $u_{0}\in(L^{\infty}\cap L^{1}\cap BV)(\R^{d})$ and $\mathbf{z}=(\mathbf{z}^{1},\mathbf{z}^{2},\mathbf{z}^{3})\in C_{0}^{0,p-var}([0,T];G^{[p]}(\R^{N_{1}+N_{2}+N_{3}}))$. If $g\ne0$ assume 
\[
|\partial_{u}^{2}f(u)|\le C_{f}<\infty,\quad\forall u\in\R
\]
for some constant $C_{f}>0$. Further, let $z^{n}=(z^{1,n},z^{2,n},z^{3,n})\in C^{1}([0,T];\R^{N_{1}+N_{2}+N_{3}})$ with $z^{n}\to\mathbf{z}$ in $p$-variation rough path metric for $n\to\infty$. Let $u^{n}$ be the unique weak entropy solution%
\footnote{Recall that we may choose $u^{n}$ right-continuous in $L_{loc}^{1}(\R^{d})$.%
} to 
\[
\partial_{t}u^{n}+\Div f(u^{n})=\nabla u^{n}\cdot H(x)\dot{z}^{1,n}+u^{n}\nu\dot{z}^{2,n}+g(x)\dot{z}^{3,n}.
\]
Then:
\begin{enumerate}
\item $(u^{n})$ is a Cauchy sequence in $L^{\infty}([0,T];L_{loc}^{1}(\R^{d}))$ with limit $u$. The limit $u$ does not depend on the particular approximating sequence $z^{n}$ and $t\mapsto u_{t}$ is right-continuous in $L_{loc}^{1}(\R^{d})$. We write 
\begin{equation}
\begin{split}du+\Div f(u)dt & =\nabla u\cdot H(x)\circ d\mathbf{z}^{1}+u\nu\circ d\mathbf{z}^{2}+g(x)\circ d\mathbf{z}^{3}\\
u(0) & =u_{0}.
\end{split}
\label{eqn:limit}
\end{equation}

\item Moreover, we have $u\in L^{\infty}([0,T]\times\R^{d})$. If $\nu,g\equiv0$ then
\[
\|u\|_{L^{\infty}([0,T]\times\R^{d})}\le\|u_{0}\|_{L^{\infty}(\R^{d})}.
\]
The function $u$ has the representation
\begin{equation}
u(t,x):=\left[e^{-\mu^{\mathbf{z}}(t)}v^{\mathbf{z}}(t,\cdot)+e^{-\mu^{\mathbf{z}}(t)}\vr^{\mathbf{z}}(t,\cdot)\right]_{|\psi^{\mathbf{z}}(t,x)}\label{eq:defn_u}
\end{equation}
where $v^{\mathbf{z}}$ is the unique weak entropy solution to 
\begin{align*}
\partial_{t}v^{\mathbf{z}}(t,x)+\Div{}^{\phi^{\mathbf{z}}}f_{\vr^{\mathbf{z}}}^{\psi^{\mathbf{z}}}(t,x,v) & =0\\
v^{\mathbf{z}}(0) & =u_{0}
\end{align*}
with 
\begin{align*}
^{\phi^{\mathbf{z}}}f_{\vr^{\mathbf{z}}}^{\psi^{\mathbf{z}}}(t,x,v) & =e^{\mu^{\mathbf{z}}(t)}{D(\psi_{t}^{\mathbf{z}})^{-1}}_{|\psi_{t}^{\mathbf{z}}(x)}f(e^{-\mu^{\mathbf{z}}(t)}(v+\vr^{\mathbf{z}}(t,x))).
\end{align*}
 
\item The solution map $(\mathbf{z},u_{0})\mapsto u$ as a mapping 
\begin{align*}
C_{0}^{0,p-var}([0,T];G^{[p]}(\R^{N}))\times(L^{\infty}\cap L^{1}\cap BV)(\R^{d}) & \to L^{\infty}([0,T]\times\R^{d})
\end{align*}
endowed with the norms
\[
\|\cdot\|_{C_{0}^{0,p-var}}\times\|\cdot\|_{L^{1}(\R^{d})}\to\|\cdot\|_{L^{\infty}([0,T];L_{loc}^{1}(\R^{d}))}
\]
is continuous on balls of initial conditions with bounded total variation and bounded $L^{\infty}$ norm.
\end{enumerate}
\end{thm}
\begin{proof}
\textbf{Step 1:} Stability for the transformed solutions

We start by proving a stability estimate on the level of the robust transformation. For smooth paths $y,z\in C^{1}([0,T];\R^{N_{1}+N_{2}+N_{3}})$ let $u^{y},u^{z}$ be the corresponding weak entropy solutions to \eqref{eq:rough_full_noise}. By Proposition \ref{prop:general well posedness} (ii) we may choose $u^{y},u^{z}$ to be right-continuous in $L_{loc}^{1}(\R^{d})$. By $\mathbf{y},$\textbf{ z} we will denote the canonical lifts of $y,z$ into geometric $p$-variation rough paths in $C_{0}^{0,p-var}([0,T];G^{[p]}(\R^{N}))$ and by $(\phi^{y},\psi^{y},\vr^{y}),(\phi^{z},\psi^{z},\vr^{z})$ the corresponding flows of diffeomorphisms introduced in the beginning of this section. Let
\[
R_{\mathbf{y,}\mathbf{z}}:=\|\mathbf{y}\|_{p-var;[0,T]}\vee\|\mathbf{z}\|_{p-var;[0,T]}
\]
and $K$ be a generic constant (i.e. it may change its value from line to line) depending on $y,z$ only via $R_{\mathbf{y,}\mathbf{z}}$, i.e. $K=K(R_{\mathbf{y,}\mathbf{z}})$ non-decreasing. The dependence on further data (such as $C_{f},\|u_{0}\|_{\infty}$) will be suppressed. From Proposition \ref{prop:full_transformation} we know that the transforms
\begin{align}
v^{y}(t,x): & =e^{\mu^{y}(t)}u^{y}(t,\psi^{y}(t,x))-\vr^{y}(t,x)\label{eq:v_def}\\
v^{z}(t,x): & =e^{\mu^{z}(t)}u^{z}(t,\psi^{z}(t,x))-\vr^{z}(t,x)\nonumber 
\end{align}
are solutions to
\begin{equation}
\partial_{t}v(t,x)+\Div^{\phi}f_{\vr}^{\psi}(t,x,v)=0\label{eq:robust}
\end{equation}
with $(\phi,\psi,\vr)=(\phi^{y},\psi^{y},\vr^{y}),(\phi^{z},\psi^{z},\vr^{z})$ respectively. From \eqref{eq:full_transf} it follows
\begin{equation}
^{\phi}F_{\vr}^{\psi}(t,x,v)\equiv0\label{eq:F_zero}
\end{equation}
since $\nabla\mu\equiv0$, due to $\nu$ being constant. For notational convenience we set
\begin{align*}
f^{y}: & ={}^{\phi^{y}}f_{\vr^{y}}^{\psi^{y}},\ f^{z}:={}^{\phi^{z}}f_{\vr^{z}}^{\psi^{z}}
\end{align*}
and we compute
\[
\div f^{y}(t,x,v)={D(\psi_{t}^{y})^{-1}}_{|\psi_{t}^{y}(x)}\dot{f}(e^{-\mu^{y}(t)}(v+\vr^{y}(t,x)))\cdot\nabla\vr^{y}(t,x)
\]
and analogously for $f^{z}$. Note that the $L^{\infty}$ bound on $u^{y}$ following from Lemma \ref{lem:linfty_bound} (and thus the one obtained for $v^{y}$ based on this) is given in terms of (cf. \eqref{eq:f_td} with $F\equiv0$)
\[
\|\td F(\cdot,\cdot,0)\|_{L^{\infty}([0,T]\times\R^{d})}=\|g\dot{y}^{3}\|_{\infty}
\]
which is unstable in $y$ in rough paths metric (similarly for $u^{z}$). Instead we need to derive an estimate on the $L^{\infty}$ norm of $v^{y},v^{z}$ based on the robust form \eqref{eq:robust}. For this we note that $f^{y},f^{z}$ satisfy $(H1)$, $(H2^{*})$ with $F\equiv0$ and to check $(H3)$ we compute 
\begin{align*}
\|\div f^{y}(\cdot,\cdot,0)\|_{L^{\infty}([0,T]\times\R^{d})} & =\|{D(\psi^{y})^{-1}}_{|\psi^{y}}\dot{f}(e^{-\mu^{y}}\vr^{y})\cdot\nabla\vr^{y}\|_{\infty}\\
 & \le\|{D(\psi^{y})^{-1}}_{|\psi^{y}}\|_{\infty}\|\dot{f}(e^{-\mu^{y}}\vr^{y})\|_{\infty}\|\nabla\vr^{y}\|_{\infty}\\
 & \le K<\infty.
\end{align*}
and%
\footnote{At this point we require the assumption $|\partial_{u}^{2}f|\le C_{f}$. If $g\equiv\text{0}$ then $\vr^{y}\equiv0$ and thus $\div f^{y}\equiv0$ so that this condition may be dropped.%
} (with $\bullet=v\in\R)$
\begin{align*}
\|\partial_{v}\div f^{y}\|_{L^{\infty}([0,T]\times\R^{d}\times\R)} & =\|e^{-\mu^{y}}{D(\psi^{y})^{-1}}_{|\psi^{y}}\ddot{f}(e^{-\mu^{y}}(\bullet+\vr^{y}))\cdot\nabla\vr^{y}\|_{\infty}\\
 & \le\|e^{-\mu^{y}}\|_{\infty}\|{D(\psi^{y})^{-1}}_{|\psi^{y}}\|_{\infty}\|\ddot{f}(e^{-\mu^{y}}(\bullet+\vr^{y}))\|_{\infty}\|\nabla\vr^{y}\|_{\infty}\\
 & \le C_{f}\|e^{-\mu^{y}}\|_{\infty}\|{D(\psi^{y})^{-1}}_{|\psi^{y}}\|_{\infty}\|\nabla\vr^{y}\|_{\infty}\\
 & \le K<\infty.
\end{align*}
Hence,
\begin{align*}
\|\div f^{y}(\cdot,\cdot,0)\|_{\infty}+\|\partial_{v}\div f^{y}\|_{\infty} & \le K<\infty
\end{align*}
and similarly for $z$ instead of $y$. From Lemma \ref{lem:linfty_bound} we conclude%
\footnote{Note that at this point \eqref{eq:F_zero} and thus $\nu$ being constant is crucial.%
}
\begin{equation}
\mcV:=\|v^{y}\|_{\infty}\vee\|v^{z}\|_{\infty}\le K<\infty\label{eq:uniform_linfty_bound}
\end{equation}
as required. Set 
\[
\O_{\mcV}:=[0,T]\times\R^{d}\times[-\mcV,\mcV].
\]
In order to apply Theorem \ref{thm:loc_entropy_stability} we first verify that the constants $\k^{*},\k_{0}^{*}$ appearing therein are bounded in terms of $K$. We observe (with $\bullet=v\in[-\mcV,\mcV])$
\begin{align*}
 & \|\partial_{v}f^{z}\|_{L^{\infty}(\O_{\mcV})}\\
 & =\|{D(\psi^{z})^{-1}}_{|\psi^{z}}\dot{f}(e^{-\mu^{z}}(\bullet+\vr^{z}))\|_{^{\infty}}\\
 & \le\|{D(\psi^{z})^{-1}}_{|\psi^{z}}\|_{\infty}\|\dot{f}(e^{-\mu^{z}}(\bullet+\vr^{z}))\|_{^{\infty}}\\
 & \le K<\infty
\end{align*}
and 
\begin{align*}
 & (2d+1)\|\nabla\partial_{v}f^{y}\|_{L^{\infty}(\O_{\mcV})}\\
 & =(2d+1)\left\Vert \nabla\left({D(\psi^{y})^{-1}}_{|\psi^{y}}\dot{f}(e^{-\mu^{y}}(\bullet+\vr^{y}))\right)\right\Vert _{^{\infty}}\\
 & \le K<\infty.
\end{align*}
Since $f^{y},f^{z}$ satisfy $(H1)$, $(H2^{*})$, $(H3)$ we may apply Theorem \ref{thm:loc_entropy_stability} (ii) to obtain
\begin{equation}
\begin{split} & \sup_{t\in[0,T]}\int_{B_{R}(x_{0})}|v^{y}(t,x)-v^{z}(t,x)|dx\\
 & \le e^{KT}\int_{B_{R+KT}(x_{0})}|u_{0}^{y}(x)-u_{0}^{z}(x)|dx\\
 & +Te^{KT}\|\partial_{v}(f^{y}-f^{z})\|_{L^{\infty}([0,T]\times B_{R+KT}(x_{0})\times[-\mcV,\mcV])}\\
 & \times\Bigg[\TV(u_{0}^{y})+C\int_{0}^{T}\int_{B_{R+KT}(x_{0})}\|\nabla\div f^{y}(t,x,\cdot)\|_{L^{\infty}([-\mcV,\mcV])}dxdt\Bigg]\\
 & +e^{KT}\int_{0}^{T}\int_{B_{R+KT}(x_{0})}\|\div(f^{y}-f^{z})(t,x\cdot)\|_{L^{\infty}([-\mcV,\mcV])}dxdt,
\end{split}
\label{eq:smooth_stability}
\end{equation}
 for all $R>0,x_{0}\in\R^{d}$. In order to bound the right hand side we note
\[
\partial_{v}f^{y}(t,x,v)={D(\psi_{t}^{y})^{-1}}_{|\psi_{t}^{y}(x)}\dot{f}(e^{-\mu_{t}^{y}}(v+\vr^{y}(t,x))).
\]
Hence, using crucially the rough paths estimates collected in Lemma \ref{lem:hoelder_rp} (with $\bullet=v\in[-\mcV,\mcV])$ 
\begin{align*}
 & \|\partial_{v}(f^{y}-f^{z})\|_{L^{\infty}([0,T]\times B_{R+KT}(x_{0})\times[-\mcV,\mcV])}\\
= & \|{D(\psi^{y})^{-1}}_{|\psi^{y}}\dot{f}(e^{-\mu^{y}}(\bullet+\vr^{y}))-{D(\psi^{z})^{-1}}_{|\psi^{z}}\dot{f}(e^{-\mu^{z}}(\bullet+\vr^{z}))\|_{\infty}\\
\le & \|{D(\psi^{y})^{-1}}_{|\psi^{y}}\|_{\infty}\|\dot{f}(e^{-\mu^{y}}(\bullet+\vr^{y}))-\dot{f}(e^{-\mu^{z}}(\bullet+\vr^{z}))\|_{\infty}\\
 & +\|{D(\psi^{y})^{-1}}_{|\psi^{y}}-{D(\psi^{z})^{-1}}_{|\psi^{z}}\|_{\infty}\|\dot{f}(e^{-\mu^{y}}(\bullet+\vr^{y}))\|_{\infty}\\
\le & KC_{f}\|(e^{-\mu^{y}}-e^{-\mu^{z}})\bullet+e^{-\mu^{y}}\vr^{y}-e^{-\mu^{z}}\vr^{z}\|_{\infty}\\
 & +K\|{D(\psi^{y})^{-1}}_{|\psi^{y}}-{D(\psi^{z})^{-1}}_{|\psi^{z}}\|_{\infty}\\
\le & K\rho_{p-var}(\mathbf{y},\mathbf{z}).
\end{align*}
Similarly,
\begin{align*}
 & \|\div(f^{y}-f^{z})\|_{L^{\infty}([0,T]\times B_{R+KT}(x_{0})\times[-\mcV,\mcV])}\\
= & \|{D(\psi^{y})^{-1}}_{|\psi^{y}}\dot{f}(e^{-\mu^{y}}(\bullet+\vr^{y}))\cdot\nabla\vr^{y}-{D(\psi^{z})^{-1}}_{|\psi^{z}}\dot{f}(e^{-\mu^{z}}(\bullet+\vr^{z}))\cdot\nabla\vr^{z}\|_{\infty}\\
\le & K\rho_{p-var}(\mathbf{y},\mathbf{z}).
\end{align*}
Due to Lemma \ref{lem:hoelder_rp} we further have (recall $K=K(R_{\mathbf{y,}\mathbf{z}})$)
\[
\|\nabla\div f^{y}\|_{L^{\infty}([0,T]\times\R^{d}\times[-\mcV,\mcV])}\le K<\infty.
\]
We obtain from \eqref{eq:smooth_stability}
\begin{align}
 & \sup_{t\in[0,T]}\int_{B_{R}}|v^{y}(t,x)-v^{z}(t,x)|dx\label{eq:local_v_stab}\\
\le & e^{KT}\int_{B_{R+KT}}|u_{0}^{y}(x)-u_{0}^{z}(x)|dx\nonumber \\
 & +Te^{KT}\rho_{p-var}(\mathbf{y},\mathbf{z})\Big(\TV(u_{0}^{y})+|B_{R+KT}(0)|\Big).\nonumber 
\end{align}
\textbf{Step 2:} Proof of (i)

Let $\mathbf{z}=(\mathbf{z}^{1},\mathbf{z}^{2},\mathbf{z}^{3})\in C_{0}^{0,p-var}([0,T];G^{[p]}(\R^{N_{1}+N_{2}+N_{3}}))$ and $z^{n}=(z^{1,n},z^{2,n},z^{3,n})\in C^{1}([0,T];\R^{N_{1}+N_{2}+N_{3}})$ with $z^{n}\to\mathbf{z}$ in $p$-variation rough path metric for $n\to\infty$. Let $u^{n}$ be the unique weak entropy solution, right-continuous in $L_{loc}^{1}(\R^{d})$, to 
\begin{align*}
\partial_{t}u^{n}+\Div f(u^{n}) & =\nabla u^{n}\cdot H(x)\dot{z}^{1,n}+u^{n}\nu\dot{z}^{2,n}+g(x)\dot{z}^{3,n}\\
u_{0}^{n} & =u_{0}.
\end{align*}
As in \eqref{eq:v_def} we define the transforms $v^{n}$, that are solutions to scalar conservation laws of the type \eqref{eq:robust}. From \eqref{eq:local_v_stab} we obtain
\begin{align*}
\sup_{t\in[0,T]}\int_{B_{R}}|v^{n}(t,x)-v^{m}(t,x)|dx\le e^{KT} & T\rho_{p-var}(\mathbf{\mathbf{z}}^{n},\mathbf{z}^{m})\Big(\TV(u_{0})+|B_{R+KT}(0)|\Big),
\end{align*}
for all $n,m\in\N$, where $K$ is a constant independent of $n,m$. In particular, the sequence $v^{n}$ is a Cauchy sequence in $L^{\infty}([0,T];L_{loc}^{1}(\R^{d}))$. Hence, there is a $v\in L^{\infty}([0,T];L_{loc}^{1}(\R^{d}))$ such that
\[
\sup_{t\in[0,T]}\int_{B_{R}}|v^{n}(t)-v(t)|dx\to0,\quad\text{for }n\to\infty,
\]
for all $R>0$. Since $t\mapsto v_{t}^{n}$ is right-continuous in $L_{loc}^{1}(\R^{d})$ so is $t\mapsto v_{t}$. It remains to be proven that this implies $L^{\infty}([0,T];L_{loc}^{1}(\R^{d}))$-convergence for $u^{n}$. Let $u$ be as in \eqref{eq:defn_u} and recall 
\[
u^{n}(t,x)=\left[e^{-\mu^{z^{n}}(t)}v^{z^{n}}(t,\cdot)+e^{-\mu^{z^{n}}(t)}\vr^{z^{n}}(t,\cdot)\right]_{|\psi^{z^{n}}(t,x)}.
\]
Since
\begin{align*}
 & u(t,\psi^{\mathbf{z}}(t,x))-u^{n}(t,\psi^{z^{n}}(t,x))\\
= & e^{-\mu^{\mathbf{z}}(t)}\left(v^{\mathbf{z}}(t,x)+\vr^{\mathbf{z}}(t,x)\right)-e^{-\mu^{z^{n}}(t)}\left(v^{z^{n}}(t,x)+\vr^{z^{n}}(t,x)\right)\\
= & e^{-\mu^{\mathbf{z}}(t)}\left(v^{\mathbf{z}}(t,x)-v^{z^{n}}(t,x)+\vr^{\mathbf{z}}(t,x)-\vr^{z^{n}}(t,x)\right)\\
 & +(e^{-\mu^{\mathbf{z}}(t)}-e^{-\mu^{z^{n}}(t)})\left(v^{z^{n}}(t,x)+\vr^{z^{n}}(t,x)\right)
\end{align*}
we have
\[
u^{n}(t,\psi^{z^{n}}(t,x))\to u(t,\psi^{\mathbf{z}}(t,x))
\]
in $L^{\infty}([0,T];L_{loc}^{1}(\R^{d}))$. Since $\psi^{z^{n}}\to\psi$ in the sense of homeomorphisms we obtain
\begin{align*}
\sup_{t\in[0,T]}\int_{B_{R}}|u^{n}(t,x)-u(t,x)|dx= & \sup_{t\in[0,T]}\int_{B_{R}}|u^{n}(t,\psi_{t}^{z^{n}}(x))-u(\psi_{t}^{z^{n}}(x))|dx\\
\le & \sup_{t\in[0,T]}\int_{B_{R}}|u^{n}(t,\psi_{t}^{z^{n}}(x))-u(\psi_{t}^{\mathbf{z}}(x))|dx\\
 & +\sup_{t\in[0,T]}\int_{B_{R}}|u(t,\psi_{t}^{z^{n}}(x))-u(t,\psi_{t}^{\mathbf{z}}(x))|dx\\
\to & 0,
\end{align*}
for $n\to\infty$ for all $R>0$ and the convergence is locally uniform with respect to $R$.

\textbf{Step 3:} Proof of (ii)

The claimed $L^{\infty}$-boundedness of $u$ follows from \eqref{eq:defn_u}, the uniform upper bound \eqref{eq:uniform_linfty_bound} and Lemma \ref{lem:hoelder_rp}. If $\nu,g\equiv0$ then $\div f^{z^{n}}\equiv0$ and it is easy to derive the claimed bound by methods similar to Lemma \ref{lem:linfty_bound}. 

\textbf{Step 4:} Proof of (iii)

Let now $\mathbf{y},\mathbf{z}\in C_{0}^{0,p-var}([0,T];G^{[p]}(\R^{N_{1}+N_{2}+N_{3}}))$ and $y^{n},z^{n}\in C^{1}([0,T];\R^{N_{1}+N_{2}+N_{3}})$ with $y^{n}\to\mathbf{y}$, $z^{n}\to\mathbf{z}$ in $p$-variation rough path metric for $n\to\infty$. From \eqref{eq:local_v_stab} we obtain
\begin{align*}
 & \sup_{t\in[0,T]}\int_{B_{R}}|v^{y^{n}}(t,x)-v^{z^{n}}(t,x)|dx\\
 & \le\int_{B_{R+MT}}|u_{0}^{1}(x)-u_{0}^{2}(x)|dx+K\rho_{p-var}(\mathbf{y}^{n},\mathbf{z}^{n})\Big(\TV(u_{0}^{1})+|B_{R+MT}(0)|\Big).
\end{align*}
Taking the limit $n\to\infty$ we obtain
\begin{align*}
 & \sup_{t\in[0,T]}\int_{B_{R}}|v^{\mathbf{y}}(t,x)-v^{\mathbf{z}}(t,x)|dx\\
 & \le\int_{B_{R+MT}}|u_{0}^{1}(x)-u_{0}^{2}(x)|dx+K\rho_{p-var}(\mathbf{y},\mathbf{z})\Big(\TV(u_{0}^{1})+|B_{R+MT}(0)|\Big),
\end{align*}
which implies the claimed local uniform continuity, but for $u^{\mathbf{y}}$ replaced by $v^{\mathbf{y}}$. Arguing as in step two this finishes the proof.
\end{proof}
As immediate consequences of the continuity of the solution mapping with respect to the driving rough path we obtain support results, large deviation results, stochastic scalar conservation laws driven by fractional Brownian motion with Hurst parameter $H$, covering the rough regime $ $$H\in(\frac{1}{4},\frac{1}{2})$. For more details on this we refer to \cite{FO13,CFO11}.

\section{Rate of convergence}

In Theorem \ref{thm:RP_limit} we have obtained the convergence $u^{n}\to u$ in $L^{\infty}([0,T];L_{loc}^{1}(\R^{d}))$ under the assumption of rough paths convergence of the driving rough paths. However, no estimate on the speed of convergence, as it would be crucial for any numerical approximation based on smoothing the noise, was derived. In this section we provide such a quantitative stability estimate. For simplicity we restrict to pure transport noise and Hölder rough paths, i.e. we consider stochastic scalar conservation laws of the type
\begin{align}
du+\Div f(u)dt & =\nabla u\cdot H(x)\circ d\mathbf{z},\label{eq:rough_transport}\\
u(0) & =u_{0}\nonumber 
\end{align}
for \textbf{$\mathbf{z}$} being a geometric $\frac{1}{p}$-Hölder rough path and $f,H$ as before. 
\begin{thm}
\label{thm:rate}For any two rough paths \textbf{$\mathbf{z}^{1},\mathbf{z}^{2}$} we let $u^{1},u^{2}$ be the corresponding solutions to \eqref{eq:rough_transport} with initial data $u_{0}^{1},u_{0}^{2}\in(L^{\infty}\cap L^{1}\cap BV)(\R^{d})$ respectively as constructed in Theorem \ref{thm:RP_limit}. For each $R>0$ there is a $K=K(R)>0$ such that
\begin{align*}
\sup_{t\in[0,T]}\|u^{1}(t)-u^{2}(t)\|_{L^{1}(\R^{d})}\le & \|u_{0}^{1}-u_{0}^{2}\|_{L^{1}(\R^{d})}+K\TV(u_{0}^{1})\rho_{p-var}(\mathbf{z}^{1},\mathbf{z}^{2}),
\end{align*}
whenever $\max_{i=1,2}\|\mathbf{z}^{i}\|_{\frac{1}{p}-H\ddot{o}l;[0,T]}\le R$.\end{thm}
\begin{proof}
Let $u^{1},u^{2}$ be the solutions to \eqref{eq:rough_transport} with initial conditions $u_{0}^{1},u_{0}^{2}$ driven by \textbf{$\mathbf{z}^{1},\mathbf{z}^{2}$ }and let
\begin{align*}
v^{1}(t,x) & =u^{1}(t,\psi_{t}^{\mathbf{z}^{1}}(x))\\
v^{2}(t,x) & =u^{2}(t,\psi_{t}^{\mathbf{z}^{2}}(x))
\end{align*}
as in \eqref{eq:v_def}. As in the proof of Theorem \ref{thm:RP_limit} we let
\[
R_{\mathbf{z}^{1},\mathbf{z}^{2}}:=\|\mathbf{z}^{1}\|_{\frac{1}{p}-H\ddot{o}l;[0,T]}\vee\|\mathbf{z}^{2}\|_{\frac{1}{p}-H\ddot{o}l;[0,T]}
\]
and $K$ be a generic constant depending only (increasingly) on $R_{\mathbf{z}^{1},\mathbf{z}^{2}}$. Again, dependence on further data will be suppressed. Moreover, we set
\begin{align*}
f^{i}(t,x,v)=f^{\psi^{\mathbf{\mathbf{z}}^{i}}}(t,x,v) & ={D(\psi_{t}^{\mathbf{z}^{i}})^{-1}}_{|\psi_{t}^{\mathbf{\mathbf{z}^{i}}}(x)}f(v)\quad i=1,2.
\end{align*}
We note that $\div f^{i}\equiv0$, $i=1,2$. Hence, with $F\equiv0$ the assumptions $(H1)$, $(H3)$ and the estimates in $(H2)$, $(H2^{*})$ are trivially satisfied. Moreover, the other regularity assumptions contained in Hypothesis \ref{hyp:H0-H2} are also easily seen to be satisfied using Lemma \ref{lem:hoelder_rp}. Taking the rough paths limit in \eqref{eq:smooth_stability} (noting $\div f^{i}\equiv0$) yields
\[
\begin{split}\int_{B_{R}(x_{0})}|v^{1}(t,x)-v^{2}(t,x)|dx\le & \int_{B_{R+MT}(x_{0})}|u_{0}^{1}(x)-u_{0}^{2}(x)|dx\\
 & +Te^{KT}\|\partial_{v}(f^{1}-f^{2})\|_{L^{\infty}([0,T]\times\R^{d}\times[-\mcV,\mcV])}\TV(u_{0}^{1}),
\end{split}
\]
for all $t\in[0,T]$. Noting
\begin{align*}
\|\partial_{v}(f^{1}-f^{2})\|_{L^{\infty}([0,T]\times\R^{d}\times[-\mcV,\mcV])}= & \|{D(\mathbf{\psi}^{\mathbf{\mathbf{\mathbf{z}}}^{1}})^{-1}}_{|\psi^{\mathbf{\mathbf{\mathbf{z}}}^{1}}}\dot{f}-{D(\psi^{\mathbf{\mathbf{\mathbf{z}}}^{2}})^{-1}}_{|\psi^{\mathbf{\mathbf{\mathbf{z}}}^{2}}}\dot{f}\|_{\infty}\\
\le & \|\dot{f}\|_{L^{\infty}([-\mcV,\mcV])}\|{D(\psi^{\mathbf{\mathbf{\mathbf{z}}}^{1}})^{-1}}_{|\psi^{\mathbf{\mathbf{\mathbf{z}}}^{1}}}-{D(\psi^{\mathbf{\mathbf{\mathbf{z}}}^{2}})^{-1}}_{|\psi^{\mathbf{\mathbf{\mathbf{z}}}^{2}}}\|_{\infty}\\
\le & K\rho_{p-var}(\mathbf{\mathbf{\mathbf{z}}}^{1},\mathbf{\mathbf{\mathbf{z}}}^{2})
\end{align*}
we obtain%
\footnote{We note that we may consider global $L^{1}$ estimates here since there is no affine-linear noise present. In order to include affine-linear noise one would have to rely on $L_{loc}^{1}$ estimates as in Theorem \ref{thm:RP_limit}.%
}
\[
\begin{split}\int_{\R^{d}}|v^{1}(t,x)-v^{2}(t,x)|dx\le & \int_{\R^{d}}|u_{0}^{1}(x)-u_{0}^{2}(x)|dx+Te^{KT}\TV(u_{0}^{1})\rho_{p-var}(\mathbf{\mathbf{\mathbf{z}}}^{1},\mathbf{\mathbf{\mathbf{z}}}^{2}),\end{split}
\]
for all $t\in[0,T]$. Hence, by $R\to\infty$ and since $\psi^{\mathbf{\mathbf{\mathbf{z}}}^{i}}$ are volume preserving flows we have
\begin{align}
\sup_{t\in[0,T]}\|u^{1}(t)-u^{2}(t)\|_{L^{1}(\R^{d})}= & \sup_{t\in[0,T]}\|v^{1}(t,(\psi_{t}^{\mathbf{\mathbf{\mathbf{z}}}^{1}}){}^{-1})-v^{2}(t,(\psi_{t}^{\mathbf{\mathbf{\mathbf{z}}}^{2}})^{-1})\|_{L^{1}(\R^{d})}\nonumber \\
\le & \sup_{t\in[0,T]}\|v^{1}(t,(\psi_{t}^{\mathbf{\mathbf{\mathbf{z}}}^{1}}){}^{-1})-v^{2}(t,(\psi_{t}^{\mathbf{\mathbf{\mathbf{z}}}^{2}})^{-1})\|_{L^{1}(\R^{d})}\nonumber \\
 & +\sup_{t\in[0,T]}\|v^{1}(t,(\psi_{t}^{\mathbf{\mathbf{\mathbf{z}}}^{2}})^{-1})-v^{2}(t,(\psi_{t}^{\mathbf{\mathbf{\mathbf{z}}}^{2}})^{-1})\|_{L^{1}(\R^{d})}\label{eq:rate_0}\\
\le & \sup_{t\in[0,T]}\|v^{1}(t,(\psi_{t}^{\mathbf{\mathbf{\mathbf{z}}}^{1}}){}^{-1})-v^{1}(t,(\psi_{t}^{\mathbf{\mathbf{\mathbf{z}}}^{2}})^{-1})\|_{L^{1}(\R^{d})}\nonumber \\
 & +\|u_{0}^{1}-u_{0}^{2}\|_{L^{1}(\R^{d})}+Te^{KT}\TV(u_{0}^{1})\rho_{p-var}(\mathbf{\mathbf{\mathbf{z}}}^{1},\mathbf{\mathbf{\mathbf{z}}}^{2}).\nonumber 
\end{align}
We now aim to estimate the first term on the right hand side. To do so, we first replace $v^{1}$ by some smooth function $v\in(L^{1}\cap BV\cap C^{1})(\R^{d})$. Carefully choosing an approximating sequence for $v^{1}$ will then yield the required estimate. Using that $ $$\psi^{\mathbf{\mathbf{\mathbf{z}}}^{1}}$ is volume preserving and setting $\Phi_{t}=(\psi_{t}^{\mathbf{\mathbf{\mathbf{z}}}^{2}})^{-1}\circ\psi_{t}^{\mathbf{\mathbf{\mathbf{z}}}^{1}}$ we observe 
\begin{align}
\|v((\psi_{t}^{\mathbf{\mathbf{\mathbf{z}}}^{1}})^{-1})-v((\psi_{t}^{\mathbf{\mathbf{\mathbf{z}}}^{2}})^{-1})\|_{L^{1}(\R^{d})} & =\|v(Id)-v((\psi_{t}^{\mathbf{\mathbf{\mathbf{z}}}^{2}})^{-1}\circ\psi_{t}^{\mathbf{\mathbf{\mathbf{z}}}^{1}})\|_{L^{1}(\R^{d})}\nonumber \\
 & =\|\sum_{i=1}^{N}v(\Phi_{t_{i+1}})-v(\Phi_{t_{i}})\|_{L^{1}(\R^{d})}\label{eq:split-up}\\
 & \le\sum_{i=1}^{N}\|v(\Phi_{t_{i+1}})-v(\Phi_{t_{i}})\|_{L^{1}(\R^{d})}\nonumber 
\end{align}
and 
\begin{align}
 & \|v(\Phi_{t_{i+1}})-v(\Phi_{t_{i}})\|_{L^{1}(\R^{d})}\nonumber \\
 & =\int_{\R^{d}}|v(\Phi_{t_{i+1}}(x))-v(\Phi_{t_{i}}(x))|dx\nonumber \\
 & \le\int_{0}^{1}\int_{\R^{d}}|\nabla v(\l\Phi_{t_{i+1}}(x)+(1-\l)\Phi_{t_{i}}(x))||\Phi_{t_{i+1}}(x)-\Phi_{t_{i}}(x)|dxd\l\label{eq:rate_1-1-1}\\
 & \le\|\Phi_{t_{i+1}}-\Phi_{t_{i}}\|_{\infty}\int_{0}^{1}\int_{\R^{d}}|\nabla v(\l\Phi_{t_{i+1}}(x)+(1-\l)\Phi_{t_{i}}(x))|dxd\l,\nonumber 
\end{align}
for any partition $0=t_{0}\le t_{1}\le\dots\le t_{N}=t$. By Lemma \ref{lem:hoelder_rp} (cf. \cite[Lemma 13]{CDFO13} for its Hölder version) we have 
\[
\|\Phi\|_{C^{\frac{1}{p}-\text{Höl}}([0,T];C^{1}(\R^{d}))}\le K\rho_{p-var}(\mathbf{\mathbf{\mathbf{z}}}^{1},\mathbf{\mathbf{\mathbf{z}}}^{2}),
\]
We now aim to prove that 
\begin{equation}
x\mapsto\l\Phi_{t_{i+1}}(x)+(1-\l)\Phi_{t_{i}}(x)\label{eq:diff}
\end{equation}
is a diffeomorphism on $\R^{d}$. Since $\Phi_{t_{i}}$ is volume preserving we have 
\[
\det\left(D\Phi_{t_{i}}(x)\right)=1.
\]
Local Lipschitz continuity of the determinant mapping then implies
\begin{align*}
\det\left(\l D\Phi_{t_{i+1}}(x)+(1-\l)D\Phi_{t_{i}}(x)\right) & =\det\left(D\Phi_{t_{i}}(x)+\l(D\Phi_{t_{i+1}}(x)-D\Phi_{t_{i}}(x))\right)\\
 & \ge1-\l K|D\Phi_{t_{i+1}}(x)-D\Phi_{t_{i}}(x)|\\
 & \ge1-K\rho_{p-var}(\mathbf{\mathbf{\mathbf{z}}}^{1},\mathbf{\mathbf{\mathbf{z}}}^{2})|t_{i+1}-t_{i}|^{\frac{1}{p}}
\end{align*}
and
\begin{align*}
\|D\left[(\Phi_{t_{i+1}}-\Phi_{t_{i}})\circ\Phi_{t_{i}}^{-1}\right]\|_{\infty} & \le\|D\Phi^{-1}\|_{C^{0}([0,T]\times\R^{d})}\|\Phi\|_{C^{\frac{1}{p}-\text{Höl}}([0,T];C^{1}(\R^{d}))}|t_{i+1}-t_{i}|^{\frac{1}{p}}\\
 & \le K\rho_{p-var}(\mathbf{\mathbf{\mathbf{z}}}^{1},\mathbf{\mathbf{\mathbf{z}}}^{2})|t_{i+1}-t_{i}|^{\frac{1}{p}}.
\end{align*}
Hence, choosing $t_{i}=T\frac{i}{N}$ with 
\[
N:=\lceil T(2K)^{p}\rho_{p-var}^{p}(\mathbf{\mathbf{\mathbf{z}}}^{1},\mathbf{\mathbf{\mathbf{z}}}^{2})\rceil
\]
we have
\begin{align}
\det\left(\l D\Phi_{t_{i+1}}(x)+(1-\l)D\Phi_{t_{i}}(x)\right) & \ge\frac{1}{2}\label{eq:det_bound}
\end{align}
and 
\begin{align}
\|D\left[(\Phi_{t_{i+1}}-\Phi_{t_{i}})\circ\Phi_{t_{i}}^{-1}\right]\|_{\infty} & \le\frac{1}{2}.\label{eq:bdd}
\end{align}
Note that \eqref{eq:diff} is injective iff $Id+\l(\Phi_{t_{i+1}}-\Phi_{t_{i}})\circ\Phi_{t_{i}}^{-1}$ is. This easily follows from \eqref{eq:bdd} which proves that \eqref{eq:diff} is a diffeomorphism. Due to \eqref{eq:det_bound} we obtain
\begin{align*}
\|v(\Phi_{t_{i+1}})-v(\Phi_{t_{i}})\|_{L^{1}(\R^{d})} & \le2\|\Phi_{t_{i+1}}-\Phi_{t_{i}}\|_{\infty}\int_{\R^{d}}|\nabla v|dx.\\
 & \le K\rho_{p-var}(\mathbf{\mathbf{\mathbf{z}}}^{1},\mathbf{\mathbf{\mathbf{z}}}^{2})|t_{i+1}-t_{i}|^{\frac{1}{p}}\int_{\R^{d}}|\nabla v|dx.
\end{align*}
Using this in \eqref{eq:split-up} yields (note that $K$ is a generic constant)
\begin{align}
\|v((\psi_{t}^{\mathbf{\mathbf{\mathbf{z}}}^{1}})^{-1})-v((\psi_{t}^{\mathbf{\mathbf{\mathbf{z}}}^{2}})^{-1})\|_{L^{1}(\R^{d})} & \le NK\rho_{p-var}(\mathbf{\mathbf{\mathbf{z}}}^{1},\mathbf{\mathbf{\mathbf{z}}}^{2})\int_{\R^{d}}|\nabla v|dx.\label{eq:rate_2-1}\\
 & \le K\rho_{p-var}(\mathbf{\mathbf{\mathbf{z}}}^{1},\mathbf{\mathbf{\mathbf{z}}}^{2})\int_{\R^{d}}|\nabla v|dx.\nonumber 
\end{align}
We now aim to choose $v^{1,n}$ to be suitable approximations of $v^{1}(t)$ so that we may pass to the limit in \eqref{eq:rate_2-1}. Theorem \ref{thm:BV-bound} allows us to estimate $\TV(v^{1}(t))$ in terms of $\TV(u_{0}^{1})$. Since we will need the right hand side, i.e. $\int_{\R^{d}}|\nabla v^{1,n}(t,x)|dx$, to be uniformly bounded in $n$, the approximations $v^{1,n}$ have to be chosen with some care. The appropriate concept is given by intermediate convergence: Due to \cite[Theorem 10.1.2]{ABM06} we may choose smooth approximations $v^{1,n}\in(L^{1}\cap BV\cap C^{\infty})(\R^{d})$ such that
\begin{align*}
v^{1,n} & \to v^{1}(t)\quad\text{in }L^{1}(\R^{d})\\
\int_{\R^{d}}|\nabla v^{1,n}|dx & \to\TV(v^{1}(t))\quad\text{for }n\to\infty.
\end{align*}
Since $\psi_{t}^{\mathbf{z}^{1}},\psi_{t}^{\mathbf{z}^{2}}$ are volume preserving, this implies $v^{1,n}((\psi_{t}^{\mathbf{z}^{1}}){}^{-1})\to v^{1}(t,(\psi_{t}^{\mathbf{z}^{1}}){}^{-1})$ and $v^{1,n}((\psi_{t}^{\mathbf{z}^{2}}){}^{-1})\to v^{1}(t,(\psi_{t}^{\mathbf{z}^{2}}){}^{-1})$ in $L^{1}(\R^{d})$. Passing to the limit in \eqref{eq:rate_2-1} we obtain
\begin{align*}
\|v^{1}(t,(\psi_{t}^{\mathbf{z}^{1}}){}^{-1})-v^{1}(t,(\psi_{t}^{\mathbf{z}^{2}})^{-1})\|_{L^{1}(\R^{d})} & \le K\rho_{p-var}(\mathbf{\mathbf{z}^{1}},\mathbf{z}^{2})\TV(v^{1}(t)).
\end{align*}
Employing Theorem \ref{thm:BV-bound} to estimate $\TV(v^{1}(t))$ in terms of $\TV(u_{0}^{1})$ and inserting in \eqref{eq:rate_0} yields 
\begin{align}
 & \sup_{t\in[0,T]}\|u^{1}(t)-u^{2}(t)\|_{L^{1}(\R^{d})}\le\|u_{0}^{1}-u_{0}^{2}\|_{L^{1}(\R^{d})}+K\TV(u_{0}^{1})\rho_{p-var}(\mathbf{\mathbf{z}^{1}},\mathbf{z}^{2}).\label{eq:local_rate}
\end{align}

\end{proof}

\appendix

\section{A transformation formula for the divergence operator\label{app:div}}

For a $C^{1}$ matrix-valued function $F\in C^{1}(\R^{d};\R^{d\times d})$ we define the divergence to act column-wise, i.e.\ 
\[
(\div F)^{j}=\div F^{j}=\sum_{i=1}^{d}\partial_{i}F_{i}^{j}.
\]
 In case of $F=D\psi$ for a function $\psi\in C^{2}(\R^{d};\R^{d})$ this means that $\div D\psi$ is the row-vector %
\[
(\div D\psi)^{j}=\sum_{i=1}^{d}\partial_{i}(D\psi)_{i}^{j}=\sum_{i=1}^{d}\partial_{i}(\partial_{j}\psi^{i})=\sum_{i=1}^{d}\partial_{j}\partial_{i}\psi^{i}=\partial_{j}\div\psi.
\]

\begin{lem}
\label{lem:div_transf} Let $g\in C^{1}(\R^{d};\R^{d})$ and $\psi\in C^{2}(\R^{d};\R^{d})$. Then
\[
\begin{split}\div(g(\psi)) & =\div((D\psi)_{\psi^{-1}}g)(\psi)-\div((D\psi)_{\psi^{-1}})(\psi)g(\psi).\end{split}
\]
\end{lem}
\begin{proof}
We compute
\[
\begin{split}\div(g(\psi)) & =\sum_{i=1}^{d}\partial_{i}(g_{i}(\psi))=\sum_{i=1}^{d}(\nabla g_{i})(\psi)\cdot\partial_{i}\psi=\sum_{i,j=1}^{d}((\partial_{j}g_{i})(\partial_{i}\psi_{j})_{\psi^{-1}})(\psi)\\
 & =\sum_{i,j=1}^{d}\partial_{j}(g_{i}(\partial_{i}\psi_{j})_{\psi^{-1}})(\psi)-\sum_{i,j=1}^{d}(g_{i}\partial_{j}(\partial_{i}\psi_{j})_{\psi^{-1}})(\psi)\\
 & =\sum_{i,j=1}^{d}\partial_{j}((\partial_{i}\psi_{j})_{\psi^{-1}}g_{i})(\psi)-\sum_{i=1}^{d}g_{i}(\psi)\sum_{j=1}^{d}(\partial_{j}(\partial_{i}\psi_{j})_{\psi^{-1}})(\psi)\\
 & =\sum_{j=1}^{d}\partial_{j}((D\psi)_{\psi^{-1}}g)_{j}(\psi)-\sum_{i=1}^{d}g_{i}(\psi)\sum_{j=1}^{d}(\partial_{j}(D\psi_{j}^{i})_{\psi^{-1}})(\psi)\\
 & =\div((D\psi)_{\psi^{-1}}g)(\psi)-\div((D\psi)_{\psi^{-1}})(\psi)g(\psi)
\end{split}
\]
 \end{proof}
\begin{prop}
\label{prop:div_transf} Let $\psi\in C^{2}(\R^{d};\R^{d})$ be a volume-preserving diffeomorphism, i.e.\ $\det D\psi=1$. Then
\[
\div((D\psi)_{|\psi^{-1}})\equiv0.
\]
 Moreover, for $g\in C^{1}(\R^{d};\R^{d})$
\[
\begin{split}\div(g(\psi)) & =\div((D\psi)_{\psi^{-1}}g)(\psi).\end{split}
\]
\end{prop}
\begin{proof}
Let $g\in C_{c}^{\infty}(\R^{d};\R^{d})$ and $\vp\in C_{c}^{\infty}(\R^{d};\R)$. We compute %
\[
\begin{split}\int_{\R^{d}}\div(g(\psi))\vp dx & =-\int_{\R^{d}}g(\psi)\cdot\nabla\vp dx\\
 & =-\int_{\R^{d}}g\cdot(\nabla\vp)(\psi^{-1})dx\\
 & =-\int_{\R^{d}}g\cdot(\nabla\vp(\psi^{-1}(\psi)))(\psi^{-1})dx.
\end{split}
\]
{} Since $\nabla\vp(\psi^{-1}(\psi))=(D\psi)^{t}(\nabla\vp(\psi^{-1}))(\psi)$ we conclude
\[
\begin{split}\int_{\R^{d}}\div(g(\psi))\vp dx & =-\int_{\R^{d}}g\cdot(D\psi)_{|\psi^{-1}}^{t}\nabla\vp(\psi^{-1})dx\\
 & =\int_{\R^{d}}\div((D\psi)_{|\psi^{-1}}g)\vp(\psi^{-1})dx\\
 & =\int_{\R^{d}}\div((D\psi)_{|\psi^{-1}}g)(\psi)\vp dx.
\end{split}
\]
 On the other hand $\div(g(\psi))=\div((D\psi)_{\psi^{-1}}g)(\psi)-\div((D\psi)_{\psi^{-1}})(\psi)g(\psi)$ by Lemma \ref{lem:div_transf}. Since $\vp$ and $g$ can be chosen arbitrarily this implies $\div((D\psi)_{\psi^{-1}})(\psi)\equiv0$.
\end{proof}

\section{Deterministic entropy solutions for hyperbolic conservation laws\label{sec:app_det_SCL}}

In this section we consider (deterministic) scalar conservation laws of the type 
\begin{equation}
\begin{split}\partial_{t}u+\Div f(t,x,u) & =F(t,x,u),\quad\text{on }[0,T]\times\R^{d}\\
u(0,x) & =u_{0}(x),\quad\text{on }\R^{d}.
\end{split}
\label{eqn:gen_SCL_app}
\end{equation}
Recall the definition of weak entropy solutions to \eqref{eqn:gen_SCL_app} from Section \ref{sec:def_not}. The main purpose of this section is the proof of a localized stability estimate for weak entropy solutions.

\subsection{Localized stability of entropy solutions to hyperbolic conservation laws}

In addition to the conditions put forward in Hypothesis \ref{hyp:H0-H2} we will require

\begin{hypothesis}\label{hyp:H4}
\begin{enumerate}
\item [($H4$)] For all $U,T>0$: 
\[
\int_{0}^{T}\int_{\R^{d}}\|(F-\div f)(t,x,\cdot)\|_{L^{\infty}([-U,U])}dxdt<\infty.
\]

\end{enumerate}
\end{hypothesis}

We now introduce some notation. For any function $u\in L^{\infty}([0,T]\times\R^{d})$ such that $t\mapsto u(t)$ is right-continuous in $L_{loc}^{1}(\R^{d})$ and $T>0$ we define 
\[
\begin{split}U_{t} & =\|u(t)\|_{L^{\infty}(\R^{d})},\\
\mcU & =\|u\|_{L^{\infty}([0,T]\times\R^{d})},\\
\mcS_{T}(u) & =\bigcup_{t\in[0,T]}\supp(u(t)),\\
\Sigma_{T}^{u} & =[0,T]\times\mcS_{T}(u)\times[-\mcU,\mcU],\\
\kappa_{0}^{*} & =(2d+1)\|\nabla\partial_{u}f\|_{L^{\infty}(\Sigma_{T}^{u};\R^{d\times d})}+\|\partial_{u}F\|_{L^{\infty}(\Sigma_{T}^{u})}.
\end{split}
\]
In the following we will assume that weak entropy solutions are right continuous as mappings from $[0,T]$ into $L_{loc}^{1}(\R^{d})$. Due to Proposition \ref{prop:general well posedness} (ii) such weak entropy solutions exist. From \cite[Theorem 2.2, Remark 2.3]{LM11} we recall
\begin{thm}
\label{thm:BV-bound}Assume $(H1)$, $(H2)$. Let $u_{0}\in(L^{\infty}\cap L^{1}\cap BV)(\R^{d})$ and let $u$ be a weak entropy solution of \eqref{eqn:gen_SCL_app}. Then $u$ satisfies $u(t)\in BV(\R^{d})$ for all $t\in[0,T]$ and 
\begin{align*}
\TV(u(t)) & \le\TV(u_{0})e^{\kappa_{0}^{*}t}\\
 & +C\int_{0}^{t}\int_{\R^{d}}e^{\kappa_{0}^{*}(t-r)}\|\nabla(F-\div f)(r,x,\cdot)\|_{L^{\infty}([-U_{r},U_{r}])}dxdr,
\end{align*}
for all $t\in[0,T]$ and some constant $C=C(d)=\frac{\pi}{2}d$.
\end{thm}
We will now recall and extend stability results for weak entropy solutions as obtained in \cite{LM11}. Let $u_{0},v_{0}\in L^{\infty}(\R^{d})$ with corresponding weak entropy solutions $u,v$ (right-continuous in $L_{loc}^{1}(\R^{d})$). We define 
\[
\begin{split}V_{t} & =\|u(t)\|_{L^{\infty}(\R^{d})}\vee\|v(t)\|_{L^{\infty}(\R^{d})}\\
\mcV & =\|u\|_{L^{\infty}([0,T]\times\R^{d})}\vee\|v\|_{L^{\infty}([0,T]\times\R^{d})}\\
\mcS_{T}(u,v) & =\bigcup_{t\in[0,T]}(\supp u(t)\cup\supp v(t))\\
\Sigma_{T}^{u,v} & =[0,T]\times\mcS_{T}(u,v)\times[-\mcV,\mcV]\\
\kappa^{*} & =\|\partial_{u}F\|_{L^{\infty}(\Sigma_{T}^{u,v})}+\|\partial_{u}\div(g-f)\|_{L^{\infty}(\Sigma_{T}^{u,v})}\\
M & =\|\partial_{u}g\|_{L^{\infty}([0,T]\times\R^{d}\times[-\mcV,\mcV]).}
\end{split}
\]

We prove a localized version of the stability estimate for scalar, inhomogeneous conservation laws obtained in \cite{LM11}. 
\begin{thm}
\label{thm:loc_entropy_stability}Let $(f,F)$, $(g,G)$ satisfy $(H1)$, $u_{0}\in(L^{\infty}\cap L^{1}\cap BV)(\R^{d}),v_{0}\in L^{\infty}(\R^{d})$ and let $u,v$ be two weak entropy solutions with respect to the initial conditions $u_{0},v_{0}$, the fluxes $f,g$ and forces $F,G$ respectively.
\begin{enumerate}
\item Suppose $(f,F)$ satisfies $(H2)$. Then 
\[
\begin{split} & \int_{B_{R}(x_{0})}|u(t,x)-v(t,x)|dx\\
 & \le e^{\kappa^{*}t}\int_{B_{R+Mt}(x_{0})}|u_{0}(x)-v_{0}(x)|dx\\
 & +\|\partial_{u}(f-g)\|_{L^{\infty}(\Sigma_{t}^{u}\cap(K_{R,M}(t,x_{0})\times\R))}\Bigg[\frac{e^{\kappa_{0}^{*}t}-e^{\kappa^{*}t}}{\kappa_{0}^{*}-\kappa^{*}}\TV(u_{0})\\
 & +C\int_{0}^{t}\frac{e^{\kappa_{0}^{*}(t-r)}-e^{\kappa^{*}(t-r)}}{\kappa_{0}^{*}-\kappa^{*}}\int_{\R^{d}}\|\nabla(F-\div f)(r,x,\cdot)\|_{L^{\infty}([-U_{r},U_{r}])}dxdr\Bigg]\\
 & +\int_{0}^{t}e^{\kappa^{*}(t-r)}\int_{B_{R+M(t-r)}(x_{0})}\|((F-G)-\div(f-g))(r,x\cdot)\|_{L^{\infty}([-V_{r},V_{r}])}dxdr,
\end{split}
\]
 for all \textup{$t\in[0,T],R>0,x_{0}\in\R^{d}$ and some constant $C=C(d)=\frac{\pi}{2}d$.}
\item Suppose $(f,F)$ satisfies $(H2^{*}),(H3)$. Then
\[
\begin{split} & \int_{B_{R}(x_{0})}|u(t,x)-v(t,x)|dx\\
 & \le e^{\kappa^{*}t}\int_{B_{R+Mt}(x_{0})}|u_{0}(x)-v_{0}(x)|dx\\
 & +\|\partial_{u}(f-g)\|_{L^{\infty}(\Sigma_{t}^{u}\cap(K_{R,M}(t,x_{0})\times\R))}\Bigg[\frac{e^{\kappa_{0}^{*}t}-e^{\kappa^{*}t}}{\kappa_{0}^{*}-\kappa^{*}}\TV(u_{0})\\
 & +C\int_{0}^{t}\frac{e^{\kappa_{0}^{*}(t-r)}-e^{\kappa^{*}(t-r)}}{\kappa_{0}^{*}-\kappa^{*}}\int_{B_{R+M(t-r)}(x_{0})}\|\nabla(F-\div f)(r,x,\cdot)\|_{L^{\infty}([-U_{r},U_{r}])}dxdr\Bigg]\\
 & +\int_{0}^{t}e^{\kappa^{*}(t-r)}\int_{B_{R+M(t-r)}(x_{0})}\|((F-G)-\div(f-g))(r,x\cdot)\|_{L^{\infty}([-V_{r},V_{r}])}dxdr,
\end{split}
\]
 for all \textup{$t\in[0,T],R>0,x_{0}\in\R^{d}$.}
\end{enumerate}
\end{thm}
\begin{proof}
(i): We are in the setting of \cite[Theorem 2.5]{LM11} except for Hypothesis \ref{hyp:H4} $(H4)$ which we do not assume. We essentially follow the same proof, except for the estimate on $K_{2}$ on page 752. We will therefore restrict to some comments on the modifications of the proof. In particular, we will employ the notations introduced in the proof of \cite[Theorem 2.5]{LM11}. We note that $\vp(r,x,s,y):=\Phi(r,x)\Psi(r-s,x-y)$, with $\Phi=\chi^{\ve}(r)\psi^{\t}(r,x)$ and $\Psi(r,x)=\nu^{\eta}(r)\mu^{\l}(x)$. Here $\chi^{\ve}(r)$, $\psi^{\t}(r,x)$ are appropriate approximations of $1_{[0,t]}$, $1_{K_{R,M}(T,x_{0})}$ and $\nu^{\eta},\mu^{\l}$ are standard Dirac sequences. In particular, we have
\begin{align*}
0 & \le\psi^{\t}(r,x)\le1_{K_{R+\t,M}(T,x_{0})}\\
0 & \le\chi^{\ve}(r)\le1_{[0,t+\ve]}
\end{align*}
and thus also 
\[
\supp\vp(\cdot,\cdot,s,y)\subseteq([0,t+\ve]\times\R^{d})\cap K_{R+\t,M}(T,x_{0}).
\]
Hence,
\begin{align*}
K_{2}\le & \int_{0}^{t+\ve+\eta}\int_{\R^{d}}\int_{\R_{+}}\|\partial_{u}(f-g)(r)\|_{L^{\infty}(\mcD\cap(B_{R+M(T-r)+\t}(x_{0})\times\R))}dr\\
 & \times\|\nabla u_{\b}(s,y)\|\nu(r-s)dydsdr\\
\le & \int_{0}^{t+\ve+\eta}\int_{\R_{+}}\|\partial_{u}(f-g)(r)\|_{L^{\infty}(\mcD\cap(B_{R+M(T-r)+\t}(x_{0})\times\R))}\\
 & \times\TV(\nabla u_{\b}(s))\nu(r-s)dsdr.
\end{align*}
All the other terms are estimated precisely as in \cite[Theorem 2.5]{LM11} and we may conclude the proof as in \cite[Theorem 2.5]{LM11}. Moreover, we note that the assumption Hypothesis \ref{hyp:H4} $(H4)$ supposed in \cite[Theorem 2.5]{LM11} is superfluous, since it is only required on balls $B_{R+M(T-r)}(x_{0})$ in the proof. On balls, however, it follows from the regularity assumptions on $F,G,f,g$ supposed in $(H1)$.

(ii): We now define a sequence of cut-off fluxes and sources: Let $\eta^{\ve}$ be a smooth cut-off function of $K_{R,M}(t,x_{0})$ satisfying
\[
1_{K_{R,M}(t,x_{0})}(r,x)\le\eta^{\ve}(r,x)\le1_{K_{R+\ve,M}(t,x_{0})}(r,x)
\]
and define 
\begin{align*}
f^{\ve}(r,x,u): & =f(r,x,u)\eta^{\ve}(r,x)\\
F^{\ve}(r,x,u): & =F(r,x,u)\eta^{\ve}(r,x)-f(r,x,u)\cdot\nabla\eta^{\ve}(r,x).
\end{align*}
Note 
\begin{equation}
F^{\ve}-\div f^{\ve}=(F-\div f)\eta^{\ve}\label{eq:cut-off_1}
\end{equation}
and thus
\begin{equation}
|F^{\ve}-\div f^{\ve}|\le|F-\div f|1_{K_{R+\ve,M}(t,x_{0})}.\label{eq:cut-off_2}
\end{equation}
Then $f^{\ve},F^{\ve}$ satisfy all the assumptions of case (i) as well as $(H3)$ with uniform bounds. By Proposition \ref{prop:general well posedness} (ii) there are unique weak entropy solutions $u^{\ve}$ to
\[
\begin{split}\partial_{t}u^{\ve}+\div f^{\ve}(t,x,u^{\ve}) & =F^{\ve}(t,x,u^{\ve})\\
u^{\ve}(0,x) & =u_{0}(x).
\end{split}
\]
The point of cutting-off $f,F$ is that now step (i) may be applied with $f,F$ replaced by $f^{\ve},F^{\ve}$. From (i) we then obtain (for $\ve>0$ small enough)
\[
\sup_{t\in[0,T]}\int_{B_{R}(x_{0})}|u^{\ve}(t,x)-u(t,x)|dx=0.
\]
From (i) we conclude 
\[
\begin{split} & \int_{B_{R}(x_{0})}|u(t,x)-v(t,x)|dx\\
 & =\int_{B_{R}(x_{0})}|u^{\ve}(t,x)-v(t,x)|dx\\
 & \le e^{\kappa^{*}t}\int_{B_{R+Mt}(x_{0})}|u_{0}(x)-v_{0}(x)|dx\\
 & +\|\partial_{u}(f^{\ve}-g)\|_{L^{\infty}(\Sigma_{t}^{u}\cap(K_{R,M}(t,x_{0})\times\R))}\Bigg[\frac{e^{\kappa_{0}^{*}t}-e^{\kappa^{*}t}}{\kappa_{0}^{*}-\kappa^{*}}\TV(u_{0})\\
 & +C\int_{0}^{t}\frac{e^{\kappa_{0}^{*}(t-r)}-e^{\kappa^{*}(t-r)}}{\kappa_{0}^{*}-\kappa^{*}}\int_{\R^{d}}\|\nabla(F^{\ve}-\div f^{\ve})(r,x,\cdot)\|_{L^{\infty}([-U_{r},U_{r}])}dxdr\Bigg]\\
 & +\int_{0}^{t}e^{\kappa^{*}(t-r)}\int_{B_{R+M(t-r)}(x_{0})}\|((F^{\ve}-G)-\div(f^{\ve}-g))(r,x\cdot)\|_{L^{\infty}([-V_{r},V_{r}])}dxdr.
\end{split}
\]
Since $f^{\ve}=f$ on $K_{R,M}(t,x_{0})$ and using \eqref{eq:cut-off_1}, \eqref{eq:cut-off_2} we obtain 
\[
\begin{split} & \int_{B_{R}(x_{0})}|u(t,x)-v(t,x)|dx\\
 & \le e^{\kappa^{*}t}\int_{B_{R+Mt}(x_{0})}|u_{0}(x)-v_{0}(x)|dx\\
 & +\|\partial_{u}(f-g)\|_{L^{\infty}(\Sigma_{t}^{u}\cap(K_{R,M}(t,x_{0})\times\R))}\Bigg[\frac{e^{\kappa_{0}^{*}t}-e^{\kappa^{*}t}}{\kappa_{0}^{*}-\kappa^{*}}\TV(u_{0})\\
 & +C\int_{0}^{t}\frac{e^{\kappa_{0}^{*}(t-r)}-e^{\kappa^{*}(t-r)}}{\kappa_{0}^{*}-\kappa^{*}}\int_{B_{R+M(t-r)+\ve}(x_{0})}\|\nabla(F-\div f)(r,x,\cdot)\|_{L^{\infty}([-U_{r},U_{r}])}dxdr\Bigg]\\
 & +\int_{0}^{t}e^{\kappa^{*}(t-r)}\int_{B_{R+M(t-r)}(x_{0})}\|((F-G)-\div(f-g))(r,x\cdot)\|_{L^{\infty}([-V_{r},V_{r}])}dxdr,
\end{split}
\]
Taking $\ve\to0$ implies the claim.
\end{proof}

\begin{cor}[Uniqueness of weak entropy solutions]
\label{cor:uniqueness_by_approx}Assume that $f,F$ satisfy $(H1)$, $(H2^{*})$ and that there is a weak entropy solution $u$ to \eqref{eqn:gen_SCL_app} obtained as the $L^{1}([0,T];L_{loc}^{1}(\R^{d}))$ limit of uniformly bounded weak entropy solutions $u^{\ve}$ to
\[
\begin{split}\partial_{t}u^{\ve}+\Div f^{\ve}(t,x,u^{\ve}) & =F^{\ve}(t,x,u^{\ve})\\
u^{\ve}(0,x) & =u_{0}^{\ve}(x)
\end{split}
\]
where $f^{\ve},F^{\ve}$ satisfy $(H1),(H2)$ and for \textup{all $R>0$ there is an $\ve>0$ such that} 
\begin{align*}
f^{\ve} & =f,\ F^{\ve}=F,\ u_{0}^{\ve}=u_{0}\quad\text{on }[0,T]\times B_{R}^{c}(0)\times\R.
\end{align*}
Then $u$ is the unique weak entropy solution to \eqref{eqn:gen_SCL_app}.\end{cor}
\begin{proof}
Let $v$ be a weak entropy solution to \eqref{eqn:gen_SCL_app}. By Theorem \ref{thm:loc_entropy_stability} (applied with $f,F=f^{\ve},F^{\ve}$ and $g,G=f,F$) we have
\[
\begin{split} & \int_{B_{R}(x_{0})}|u^{\ve}(t,x)-v(t,x)|dx\\
 & \le e^{\kappa^{*}t}\int_{B_{R+Mt}(x_{0})}|u_{0}^{\ve}(x)-v_{0}(x)|dx\\
 & +\|\partial_{u}(f^{\ve}-f)\|_{L^{\infty}(K_{R,M}(t,x_{0})\times\R)}\Bigg[\frac{e^{\kappa_{0}^{*}t}-e^{\kappa^{*}t}}{\kappa_{0}^{*}-\kappa^{*}}\TV(u_{0}^{\ve})\\
 & +C\int_{0}^{t}\frac{e^{\kappa_{0}^{*}(t-r)}-e^{\kappa^{*}(t-r)}}{\kappa_{0}^{*}-\kappa^{*}}\int_{\R^{d}}\|\nabla(F^{\ve}-\div f^{\ve})(r,x,\cdot)\|_{L^{\infty}([-U_{r},U_{r}])}dxdr\Bigg]\\
 & +\int_{0}^{t}e^{\kappa^{*}(t-r)}\int_{B_{R+M(t-r)}(x_{0})}\|((F^{\ve}-F)-\div(f^{\ve}-f))(r,x\cdot)\|_{L^{\infty}([-V_{r},V_{r}])}dxdr.,
\end{split}
\]
with coefficients $\k_{0}^{*},\k^{*}$ possibly depending on $\ve$. Note that $M$, however, is independent of $\ve$ since $\sup_{\ve>0}\|u^{\ve}\|_{L^{\infty}([0,T]\times\R^{d})}\le C$. Choosing $\ve>0$ small enough we obtain
\[
\begin{split}\int_{0}^{T}\int_{B_{R}(x_{0})}|u^{\ve}(t,x)-v(t,x)|dxdt\le & 0.\end{split}
\]

\end{proof}
Condition $(H3)$ in Proposition \ref{prop:general well posedness} is required in order to obtain uniform bounds on the vanishing viscosity approximants used to construct weak entropy solutions. Since we will require uniform control on the $L^{\infty}$ norm of weak entropy solutions we note
\begin{lem}
\label{lem:linfty_bound}Assume that $f$, $F$ satisfy $(H1)$, $(H2^{*})$, $(H3)$, let $u_{0}\in(L^{\infty}\cap L^{1}\cap BV)(\R^{d})$, $u$ be the corresponding weak entropy solution to \eqref{eqn:gen_SCL_app} and define
\[
M:=\|(\div f-F)(\cdot,\cdot,0)\|_{L^{\infty}([0,T]\times\R^{d})}+\|\partial_{u}(\div f-F)\|_{L^{\infty}([0,T]\times\R^{d}\times\R)}.
\]
Then
\[
\|u\|_{L^{\infty}([0,T]\times\R^{d})}\le(\|u_{0}\|_{L^{\infty}(\R^{d})}+1)e^{2MT}.
\]
\end{lem}
\begin{proof}
The weak entropy solution $u$ is constructed in \cite{K70} by first cutting-off $f,F$, then mollifying the coefficients and then applying a vanishing viscosity approximation. Since the conditions $(H1)$, $(H2^{*})$, $(H3)$ are preserved (with uniform bounds) under these cut-off and mollification procedures, it is enough to prove the claimed uniform bound on the level of the vanishing viscosity approximations
\begin{equation}
\begin{split}\partial_{t}u^{\ve}+\Div f(t,x,u^{\ve}) & =\ve\D u^{\ve}+F(t,x,u^{\ve})\\
u^{\ve}(0,x) & =u_{0}(x).
\end{split}
\label{eq:vanish_visc}
\end{equation}
Since comparison holds for \eqref{eq:vanish_visc} it is sufficient to construct appropriate sub- and supersolutions. For this we rewrite \eqref{eq:vanish_visc} in the form
\[
\partial_{t}u^{\ve}+\partial_{u}f(t,x,u^{\ve})\nabla u^{\ve}=\ve\D u^{\ve}+(F-\div f)(t,x,u^{\ve}).
\]
We set $M_{0}:=\|u_{0}\|_{\infty}$ and 
\[
K(t):=(M_{0}+1)e^{2Mt}.
\]
Then
\[
\partial_{u}f(t,x,K)\nabla K=\ve\D K=0
\]
and
\begin{align*}
(F-\div f)(t,x,K) & =(F-\div f)(t,x,K)-(F-\div f)(t,x,0)+(F-\div f)(t,x,0)\\
 & =\int_{0}^{K}\partial_{u}(F-\div f)(t,x,u)du+(F-\div f)(t,x,0)\\
 & \le(K+1)M.
\end{align*}
Since 
\begin{align*}
\partial_{t}K & =2M(M_{0}+1)e^{2Mt}\\
 & =M(M_{0}+1)e^{2Mt}+M(M_{0}+1)e^{2Mt}\\
 & \ge(K+1)M,
\end{align*}
we observe that $K$ is a supersolution to \eqref{eq:vanish_visc}. The construction of a subsolution proceeds analogously.
\end{proof}

\bibliographystyle{amsalpha.bst}
\bibliography{../../latex-refs/refs}

\end{document}